\documentclass[a4paper, 11pt, reqno]{amsart}
\usepackage{amsmath,amsthm,amssymb,amscd}
\usepackage[english]{babel}
\usepackage[ansinew]{inputenc}
\usepackage[T1]{fontenc}
\usepackage{hyperref}

\usepackage[all]{xy}
\usepackage{xspace}
\usepackage{mathrsfs}
\usepackage{color}
\usepackage{epsfig}
\usepackage{float}
\usepackage{wasysym}
\usepackage{setspace} 
\usepackage{enumerate}
\usepackage{mathdots}

\usepackage{dsfont}

\renewcommand{\sf}{\mathsf}

\newcommand{\diam}{\mathrm{diam}}

\newcommand{\I}{\mathds{1}}
\newcommand{\N}{{\mathbb N}}
\newcommand{\Z}{{\mathbb Z}}

\newcommand{\R}{{\mathbb R}}

\newcommand{\Lp}{{\ell}}
\newcommand{\supp}{{\rm supp}}

\def\Xint#1{\mathchoice
{\XXint\displaystyle\textstyle{#1}}%
{\XXint\textstyle\scriptstyle{#1}}%
{\XXint\scriptstyle\scriptscriptstyle{#1}}%
{\XXint\scriptscriptstyle\scriptscriptstyle{#1}}%
\!\int}
\def\XXint#1#2#3{{\setbox0=\hbox{$#1{#2#3}{\int}$ }
\vcenter{\hbox{$#2#3$ }}\kern-.6\wd0}}

\def\dashint{\Xint-}

\newcommand{\be}{\begin{enumerate}}
\newcommand{\ee}{\end{enumerate}}

\numberwithin{figure}{section}

\newtheorem{theorem}{Theorem}[section]
\newtheorem{lemma}[theorem]{Lemma}
\newtheorem{proposition}[theorem]{Proposition}
\newtheorem{corollary}[theorem]{Corollary}
\newtheorem{definition}{Definition}[section]

\newenvironment{remark}[1][Remark.]{\begin{trivlist}
\item[\hskip \labelsep \textsc{#1}]}{\end{trivlist}}

\makeatletter

\@addtoreset{equation}{section}
\makeatother

\title[Orlicz spaces and Heintze groups]{Orlicz spaces and the large scale geometry of Heintze groups}
\author{Matias Carrasco Piaggio}
\email{matias@math.u-psud.fr}
\address{Laboratoire de Math\'ematiques d'Orsay}

\begin{document}

\maketitle
\begin{abstract}
We consider an Orlicz space based cohomology for metric (measured) spaces with bounded geometry. We prove the quasi-isometry invariance for a general Young function. In the hyperbolic case, we prove that the degree one cohomology can be identified with an Orlicz-Besov function space on the boundary at infinity. We give some applications to the large scale geometry of homogeneous spaces with negative curvature (Heintze groups). As our main result, we prove that if the Heintze group is not of Carnot type, any self quasi-isometry fixes a distinguished point on the boundary and preserves a certain foliation on the complement of that point.
\end{abstract}
\begin{quote}
\end{quote}
\begin{quote}
\footnotesize{\textsc{2010 Subject classification}: 20F67, 30Lxx, 46E30, 53C30.}
\end{quote}

\section{Introduction}\label{intro}

\noindent In this article we are interested in the large scale geometry of Heintze groups. Homogeneous manifolds with negative sectional curvature where characterized by Heintze in \cite{Heintze74}. Each such manifold is isometric to a solvable Lie group $X_\alpha$ with a left invariant metric, and the group $X_\alpha$ is a semi-direct product $N\rtimes_\alpha \R$ where $N$ is a connected, simply connected, nilpotent Lie group, and $\alpha$ is a derivation of $N$ whose eigenvalues all have positive real parts. Such a group is called a \emph{Heintze group}.

\medskip
\noindent A \emph{purely real} Heintze group is a Heintze group $X_\alpha$ as above, for which the action of $\alpha$ on the Lie algebra of $N$ has only real eigenvalues. Every Heintze group is quasi-isometric to a purely real Heintze group, unique up to isomorphism, see \cite[Section 5B]{Cornulier13}. In the sequel we will focus only on purely real Heintze groups.

\medskip
\noindent Let $\mathfrak{n}$ be the Lie algebra of $N$, and $\mathrm{Der}(\mathfrak{n})$ be the Lie algebra of derivations on $\mathfrak{n}$. We denote by $\mathrm{Exp}:\mathrm{Der}(\mathfrak{n})\to\mathrm{Aut}(\mathfrak{n})$ the matrix exponential. The group structure of $X_\alpha$ is then given by a contracting action $\tau:\R\to \text{Aut}(N)$, where $\tau$ satisfies $d_e\tau(t)=\mathrm{Exp}(-t\alpha)$.

\medskip
\noindent We will use the notation $(x,t)$ to denote a point of $X_\alpha$. Any left invariant metric on $X_\alpha$ is Gromov hyperbolic, since any two such metrics are bi-Lipschitz equivalent. Assume that $X_\alpha$ is equipped with a left invariant metric for which the vertical lines $t\mapsto (x,t)$ are unit-speed geodesics. They are all asymptotic when $t\to -\infty$, and hence, they define a ``special'' boundary point denoted $\infty$. The boundary at infinity $\partial X_{\alpha}$ is a topological $n$-sphere, and can be therefore identified with the one-point compactification $N\cup\{\infty\}$. The left action of $X_\alpha$ on its boundary has two orbits, namely, $N$ and $\infty$.

\medskip
\noindent Two general problems motivate this work: first, to understand $\mathrm{QIsom}\left(X_\alpha\right)$, the group of self quasi-isometries of $X_\alpha$; and second, the quasi-isometric classification of Heintze groups. These problems have been approached by many authors and by means of several methods, see for instance \cite{UrsulaHamenstadt87,PierrePansu89d,FarbMosher00,PierrePansu07,TulliaDymarz10,Cornulier-Tessera11,DymarzPeng11,IrinePeng11,XiangdongXie12,NageswariXie12,XiangdongXie14}. We refer the reader to \cite{Cornulier13} for a survey on the subject.

\medskip
\noindent In this article, we focus on the \emph{Pointed Sphere Conjecture} \cite[Conjecture 6.\,C.\,9]{Cornulier13}. It states that $\infty$ is fixed by the boundary homeomorphism of any self quasi-isometry of $X_\alpha$, unless $X_\alpha$ is isometric (for some left invariant metric) to a rank one symmetric space. 

\medskip
\noindent The conjecture is known to be true in the following cases. Recall that $X_\alpha$ is said to be of \emph{Carnot type} if the Lie algebra spanned by the eigenvectors corresponding to the smallest eigenvalue of $\alpha$ is the whole algebra $\mathfrak{n}$ \cite[Definition 2.\,G.\,1]{Cornulier13}.
\begin{itemize}
\item \cite[Corollary 6.\,9]{PierrePansu89d}, the conjecture holds whenever $X_\alpha$ is not of Carnot type and $\alpha$ is diagonalizable.
\item \cite{XiangdongXie12,NageswariXie12,XiangdongXie14}, the conjecture holds when $N\simeq \R^n$ is abelian. 
\item \cite{XiangdongXie14c}, the conjecture holds when $N\simeq H_{2n+1}$ is the real Heisenberg group of dimension $2n+1$ and $\alpha$ is diagonalizable. Also in \cite{XiangdongXie14b}, the conjecture is shown to be true for a non-diagonalizable derivation when $n=1$.
\end{itemize}
The idea behind the proofs is similar in all the three cases. It consists in finding a quasi-isometry invariant foliation on $\partial X_\alpha$ which is singular at the point $\infty$. The leaves of this foliation are the accessibility classes of points by rectifiable curves, with respect to an appropriate visual metric on the boundary. To this end, Pansu consider $L^p$-cohomology, and Xie consider the $p$-variation of functions on the boundary.

\medskip
\noindent Following an original idea of Romain Tessera, and Pansu's methods, we propose here an approach based on the theory of Orlicz spaces \cite{Rao-Ren91}. This allows us to extend these results to all Heintze groups which are not of Carnot type. We mention that Orlicz spaces based cohomologies have been considered recently also in \cite{Kopylov13,Kopylov-Panenko13}.

\medskip
\noindent The next theorem is our main result. Let $\mu_1$ be the smallest eigenvalue of $\alpha$, and let $H_1$ be the closed connected subgroup of $N$ whose Lie algebra is spanned by the $\mu_1$-eigenvectors belonging to the $\mu_1$-Jordan blocks of maximal size. It is a non-trivial and proper subgroup of $N$ when $X_\alpha$ is not of Carnot type, see Section \ref{resultsHeintze} for more details.

\begin{theorem}\label{fixpoint}
Let $X_\alpha$ be a purely real Heintze group that is not of Carnot type. The boundary homeomorphism of any self quasi-isometry of $X_\alpha$ fixes the special boundary point $\infty$. Furthermore, in this case, it preserves the left cosets of the subgroup $H_1$.
\end{theorem}

\noindent Another important ingredient in our approach is a localization technique. Let us motivate it by considering the following examples. Let $N=\R^2$ and consider the Heintze group $X_i:=X_{\alpha_i}$, $i=1,2,3$, where 
$$\mathrm{(i)}\ \alpha_1=\left(\begin{tabular}{c c} 1 & 0 \\ 0 & 1 \end{tabular}\right)\ \mathrm{(ii)}\ \alpha_2=\left(\begin{tabular}{c c} 1 & 0 \\ 0 & $\mu$ \end{tabular}\right)\ \mathrm{(iii)}\ \alpha_3=\left(\begin{tabular}{c c} 1 & 1 \\ 0 & 1 \end{tabular}\right),$$
and $\mu>1$. Notice that $X_1$ is isometric to the real hyperbolic space $\mathbb{H}^3$. The degree one $L^p$-cohomology of $X_i$ can be identified with a Besov space on the boundary \cite{PierrePansu89,Bourdon-Pajot04}. Let us consider the quasi-isometry invariant Banach algebra of continuous Besov functions $A^p(\partial X_i)$. The dependence on $p$ of this algebra is summarized in Figure \ref{fig:global}.

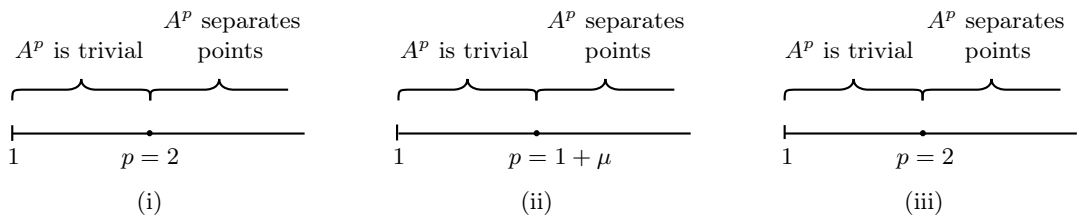
\begin{figure}[h]
\centering
\setlength{\unitlength}{1cm}
\begin{picture}(14.5,2.7)
\includegraphics[width=14.5cm]{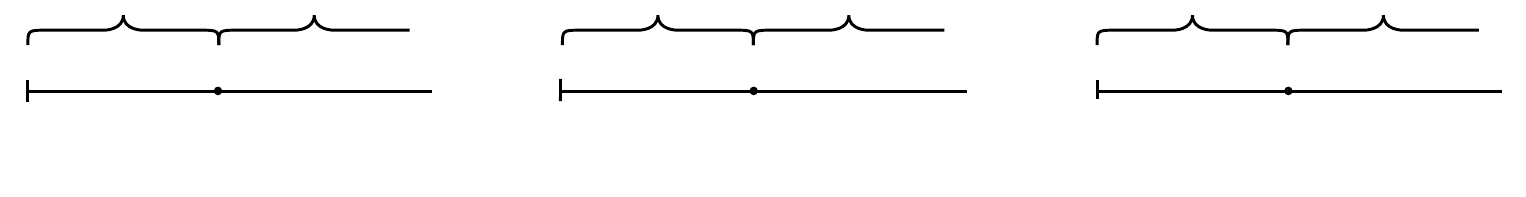}
\put(-14.3,0.6){\footnotesize $1$}
\put(-9.23,0.6){\footnotesize $1$}
\put(-4.12,0.6){\footnotesize $1$}
\put(-14.2,2){\footnotesize $A^p$ is trivial}
\put(-12.25,2.4){\footnotesize $A^p$ separates}
\put(-11.8,2){\footnotesize points}
\put(-9.13,2){\footnotesize $A^p$ is trivial}
\put(-7.2,2.4){\footnotesize $A^p$ separates}
\put(-6.7,2){\footnotesize points}
\put(-4.07,2){\footnotesize $A^p$ is trivial}
\put(-2.2,2.4){\footnotesize $A^p$ separates}
\put(-1.7,2){\footnotesize points}
\put(-12.8,0.6){\footnotesize $p=2$}
\put(-2.6,0.6){\footnotesize $p=2$}
\put(-7.7,0.6){\footnotesize $p=1+\mu$}
\put(-12.6,0){\footnotesize (i)}
\put(-7.55,0){\footnotesize (ii)}
\put(-2.5,0){\footnotesize (iii)}
\end{picture}
\caption{Dependence on $p$ of $A_p(\partial X_i)$.\label{fig:global}}
\end{figure}

\medskip
\noindent One way to isolate the point $\infty$ is suggested by the following construction which appears in \cite{VladimirShchur14}. Let $Z_i$ be the \emph{Heintze cone} defined as the quotient space of $X_i$ by the discrete group of translations $\Z^2$. The boundary $\partial Z_i$ is the union of a torus $T_i$ and the isolated point $\infty$. The dependence on $p$ of $A^p(\partial Z_i)$ is summarized in Figure \ref{fig:local}. One would like to define a similar cone with respect to any other point $\xi\in \partial X_i$.

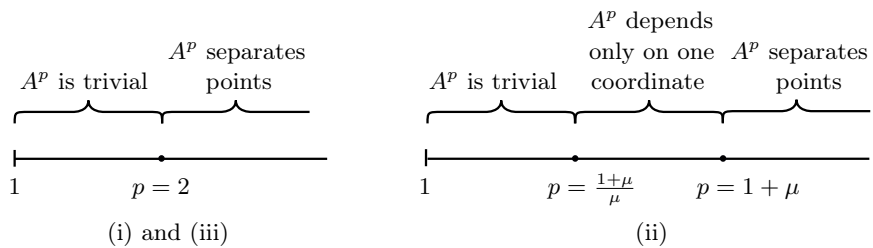
\begin{figure}
\centering
\setlength{\unitlength}{1cm}
\begin{picture}(12,3.1)
\includegraphics[width=12cm]{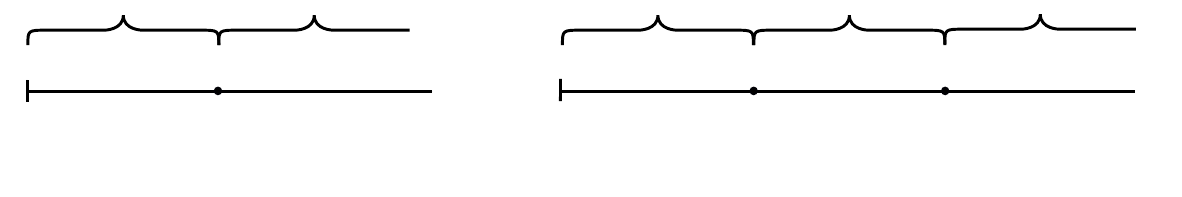}
\put(-11.8,0.6){\footnotesize $1$}
\put(-6.4,0.6){\footnotesize $1$}
\put(-11.65,2){\footnotesize $A^p$ is trivial}
\put(-9.7,2.4){\footnotesize $A^p$ separates}
\put(-9.2,2){\footnotesize points}
\put(-6.27,2){\footnotesize $A^p$ is trivial}
\put(-2.3,2.4){\footnotesize $A^p$ separates}
\put(-1.7,2){\footnotesize points}
\put(-4.13,2){\footnotesize coordinate}
\put(-4.16,2.4){\footnotesize only on one}
\put(-4.18,2.8){\footnotesize $A^p$ depends}
\put(-10.17,0.6){\footnotesize $p=2$}
\put(-2.74,0.6){\footnotesize $p=1+\mu$}
\put(-4.7,0.6){\footnotesize $p=\frac{1+\mu}{\mu}$}
\put(-10.5,0){\footnotesize (i) and (iii)}
\put(-3.55,0){\footnotesize (ii)}
\end{picture}
\caption{Dependence on $p$ of $A_p(\partial Z_i)$.\label{fig:local}}
\end{figure}

\medskip
\noindent The pull-back of a function $u\in A^p(\partial Z_i)$ by the projection map is periodic and does not define an element of $A^p(\partial X_i)$. Nevertheless, it satisfies a local integrability condition, i.e. it belongs to the Fr\'echet algebra $A^p_\mathrm{loc}(\partial X_i\setminus \{\infty\})$. The key point is that the dependence on $p$ of the algebras $A^p_\mathrm{loc}(\partial X_i\setminus \{\xi\})$ coincides with that summarized in Figure \ref{fig:local} when $\xi=\infty$, and with that in Figure \ref{fig:global} when $\xi\in\R^2$. This explains the case (ii).

\medskip
\noindent Let us look at (iii) in more detail. The parabolic visual metric on $\partial X_3\setminus\{\infty\}=\R^2$ is bi-Lipschitz equivalent to the function
$$\varrho\left((x_1,y_1);(x_2,y_2)\right)=\max\left\{|y_2-y_1|,\left|(x_2-x_1)-(y_2-y_1)\log|y_2-y_1|\right|\right\}.$$
From this expression, we see that the projection functions, $\pi_1(x,y)=x$ and $\pi_2(x,y)=y$, have different regularity properties with respect $\varrho$. That is, $\pi_2$ is Lipschitz, while $\pi_1$ satisfies the inequality
\begin{equation}\label{log}
\left|\pi_1(v)-\pi_1(w)\right|\lesssim \varrho(v,w)\log\left(\frac{1}{\varrho(v,w)}\right).
\end{equation}
In particular, $\pi_1$ is $\alpha$-H\"older for any exponent $\alpha\in (0,1)$. Therefore, both projections belong to $A^{2/\alpha}_{\mathrm{loc}}(\partial X_3\setminus \{\infty\})$ for any $\alpha<1$, which explains Figure \ref{fig:local}.(iii). In other words, the $L^p$-cohomology in degree one is not sensible to the logarithmic term appearing in (\ref{log}). As we will show later, a well chosen Orlicz-Besov space detects the difference between $\pi_1$ and $\pi_2$, and we are able to distinguish $\infty$ in case (iii) also. 

\medskip
\noindent The quasi-isometry invariance of Orlicz-Besov spaces is not evident at first sight, we show it for a class of Young functions which is sufficient for our purposes in Sections \ref{qiinvariance} and \ref{boundary}. We deal with localization in Section \ref{localization}.

\medskip
\noindent As a by product, finer results regarding $\mathrm{QIsom}(X_\alpha)$ are obtained. Let us outline them.

\subsection{On the Orlicz cohomology of a hyperbolic complex and its localization} 

Let $X$ be a finite dimensional simplicial complex with bounded geometry. That is, there exists a constant $\sf{N}$ such that any vertex of $X$ is contained in at most $\sf{N}$ simplices. Suppose $X$ is equipped with a geodesic distance making each of its simplices isometric to a standard Euclidean simplex. We further assume that $X$ is uniformly contractible: it is contractible and any ball $B(x,r)$ in $X$ is contractible in the ball $B(x,r')$, for some $r'\geq r$ which depends only on $r$. In the hyperbolic case, this condition is not restrictive.

\medskip
\noindent By a Young function we mean an even convex function $\phi:\R\to \R_+$, with $\phi(0)=0$ and $\lim_{t\to\infty}\phi(t)=+\infty$. For such a function, we introduce here the Orlicz cohomology, denoted by $\ell^\phi H^\bullet(X)$, of the complex $X$. It consists on a direct generalization of the ordinary $\ell^p$-cohomology introduced in \cite{Bourdon-Pajot04,MikhailGromov93,PierrePansu89}, where $\phi(t)=\phi_p(t):=|t|^p$. As in the ordinary case, we show that it is a quasi-isometry invariant of $X$.

\medskip
\noindent Suppose that $X$ is in addition a quasi-starlike Gromov-hyperbolic space. Recall that quasi-starlike means that any point of $X$ is at uniformly bounded distance from some geodesic ray. Inspired by the works of Pansu \cite{PierrePansu89,PierrePansu02,PierrePansu08}, and Bourdon-Pajot \cite{Bourdon-Pajot04}, we identify the degree one Orlicz cohomology of $X$ with an Orlicz-Besov functional space on its boundary $\partial X$. To this end, we need to assume a decay condition on the Young function.

\begin{definition}\label{growth}
Let $\phi$ be a Young function. We say that $\phi$ is \emph{doubling}, if there exist $t_0>0$ and $\sf{K}\geq 2$ such that $\phi(2t)\leq \sf{K}\phi(t)\ \text{ for all } t\in[0,t_0]$.
\end{definition}

\noindent The doubling condition in the above definition is known in the literature as the $\Delta_2(0)$ condition. It admits several equivalent formulations, see for example \cite[Thm.\,3 Ch.\,2]{Rao-Ren91}.

\medskip 
\noindent Moreover, suppose that the boundary at infinity $\partial X$ admits a visual metric $\varrho$ which is Ahlfors regular of dimension $Q>0$. That is, the $Q$-dimensional Hausdorff measure $H$ of a ball of radius $r\leq \diam\,\partial X$ is comparable to $r^Q$. We refer the reader to \cite{Bridson-Haefliger99,Ghys-delaHarpe88,Heinonen01} for basic background.

\medskip
\noindent In the space of pairs $\partial^2 X$, consider the measure 
$$d\lambda(\xi,\zeta)=\frac{dH\otimes dH(\xi,\zeta)}{\varrho(\xi,\zeta)^{2Q}}.$$
If $u:\partial X\to\R$ is a measurable function, we define its Orlicz-Besov $\phi$-norm as
\begin{equation}\label{besovnorm}\langle u\rangle_\phi:=\inf\left\{\alpha>0:\int_{\partial^2X}\phi\left(\frac{u(\xi)-u(\zeta)}{\alpha}\right)d\lambda(\xi,\zeta)\leq 1\right\}.\end{equation}
Then, the Orlicz-Besov space is by definition
\begin{equation}\label{orliczbesov}
B^\phi\left(\partial X,\varrho\right):=\left\{u:\partial X\to \R\text{ measurable}: \langle u\rangle_\phi<\infty\right\}.
\end{equation}
We denote by $\R$ the functions on $\partial X$ which are constant $H$-almost everywhere. Then the space $B^\phi\left(\partial X,\varrho\right)/\R$ when equipped with the norm $\langle \cdot\rangle_\phi$ is a Banach space.

\begin{theorem}\label{thetracemap} Let $\phi$ be a doubling Young function, and $\varrho$ be an Ahlfors regular visual metric on $\partial X$. There exists a canonical isomorphism of Banach spaces between $\ell^\phi H^1(X)$ and $B^\phi(\partial X,\varrho)/\R$. In particular, $\ell^\phi H^1(X)$ is reduced.
\end{theorem}

\medskip
\noindent Notice that by \cite[Prop. 2.1]{Bourdon-Pajot04}, any Ahlfors regular compact metric space $Z$ is bi-Lipschitz homeomorphic to $\partial X$ for some geometric hyperbolic complex $X$ as above, and the quasi-isometry class of $X$ depends only on the quasi-symmetry class of $Z$. In particular, Orlicz-Besov spaces are quasisymmetry invariants of $Z$. When the metric space $Z$ is the Euclidean $n$-sphere, and $\phi=\phi_p$, this Orlicz-Besov space coincides with the classical Besov space $B^{n/p}_{p,p}$ \cite{HansTriebel83}.

\medskip
\noindent In order to quantify the influence of a point $\xi\in\partial X$ to the nullity of the cohomology spaces, we introduce a local version of the $\ell^\phi$-cohomology. This provides us with an interesting tool capable to distinguish local features of the quasiconformal geometry of the boundary.

\medskip
\noindent Given a quasi-isometric embedding $\iota:Y\to X$, where $Y$ is another hyperbolic simplicial complex with bounded geometry, one can define the pull-back of cochains $\iota^*$, see for example \cite{Bourdon-Pajot04}. For $\xi\in \partial X$, denote by $\mathcal{Y}(X,\xi)$ the collection of all such quasi-isometric embeddings with $\iota(\partial Y)\subset \partial_\xi X$. We define the space of locally $\phi$-integrable $k$-cochains of $X$ (with respect to $\xi$) as
\begin{align*}\Lp^\phi_{\mathrm{loc}}(X_k,\xi):=&\big\{\tau:X_k\to\R:\iota^*(\tau)\in \Lp^\phi(Y_k),\forall\,\iota\in\mathcal{Y}(X,\xi)\big\}.
\end{align*}
Here $X_k$ and $Y_k$ denote the set of $k$-simplices of $X$ and $Y$ respectively. We emphasize the fact that the cochains are defined globally, but the integrability condition is ``local''. As we will show in Section \ref{localization}, $\Lp^\phi_{\mathrm{loc}}(X_k,\xi)$ is a Fr\'echet space and the coboundary operators
$$\delta_k:\Lp^\phi_{\mathrm{loc}}(X_k,\xi)\to \Lp^\phi_{\mathrm{loc}}(X_{k+1},\xi)$$
are Lipschitz continuous.

\begin{definition}\label{deflocal}
Consider a point $\xi\in \partial X$. We define the local $\Lp^\phi$-cohomology of $X$ with respect to $\xi$ as
$$\Lp^\phi_{\mathrm{loc}} H^k(X,\xi):=\mathrm{ker}\,\delta_{k}/\mathrm{im}\,\delta_{k-1}.$$
The reduced local $\Lp^\phi$-cohomology is defined as usual taking the quotient by $\overline{\mathrm{im}\,\delta_{k-1}}$.
\end{definition}

\noindent Notice that by definition, the local $\Lp^\phi$-cohomology is a quasi-isometry invariant of the pair $(X,\xi)$. More precisely, if $F:X\to X'$ is a quasi-isometry between two hyperbolic simplicial complexes as above, then $\Lp^\phi_{\mathrm{loc}}(X,\xi)$ is isomorphic as a topological vector space to $\Lp^\phi_{\mathrm{loc}}(X',F(\xi))$. We denote also by $F$ the boundary extension.

\medskip
\noindent As for the global cohomology, we can identify the local Orlicz cohomology in degree one with a local Orlicz-Besov space on the parabolic boundary $\partial_\xi X:=\partial X\setminus\{\xi\}$. We suppose that $\partial_\xi X$ is equipped with an Ahlfors regular parabolic visual metric $\varrho_\xi$. We define the local Orlicz-Besov space on $\partial_\xi X$ as 
$$B^\phi_{\mathrm{loc}}\left(\partial_\xi X,\varrho_\xi\right):=\big{\{}u:\partial_\xi X\to\R:\ \langle u \rangle_{\phi,\xi,K}<\infty,\forall\text{ compact } K\subset\partial_\xi X\big{\}},$$
where $\langle \cdot \rangle_{\phi,\xi,K}$ is the seminorm defined as in (\ref{besovnorm}), but replacing $\varrho$ by $\varrho_\xi$ and integrating over the compact $K$.

\begin{theorem}\label{thelocaltracemap} Let $\phi$ be a doubling Young function, and $\varrho_\xi$ be an Ahlfors regular parabolic visual metric on $\partial_\xi X$. There exists a canonical isomorphism of Fr\'echet spaces between $\ell^\phi_{\mathrm{loc}} H^1(X,\xi)$ and $B^\phi_{\mathrm{loc}}(\partial_\xi X,\varrho_\xi)/\R$. In particular, $\ell^\phi_{\mathrm{loc}} H^1(X,\xi)$ is reduced.
\end{theorem}

\noindent These identifications are very useful to define several quasi-isometry invariants. Consider a family of doubling Young functions $\{\phi_i:i\in I\}$ indexed on a totally ordered set $I$, and suppose that it is non-decreasing in the sense that if $i\leq j$ in $I$, then $\phi_i\preceq\phi_j$ (see Section \ref{preOrlicz} for the definition of the relation $\preceq$). We define the critical exponent of $I$ as
\begin{equation}\label{exponent}
p_{\neq 0}(X,I):=\inf\left\{i\in I:\ell^{\phi_i}H^1(X)\neq 0\right\}.
\end{equation}
Notice that by Theorem \ref{thetracemap}, there is a canonical inclusion $\ell^{\phi_1}H^1(X)\subset \ell^{\phi_2}H^1(X)$ whenever $\phi_1\preceq \phi_2$. When the family of Young functions is given by $\{\phi_p:p\in[1,+\infty)\}$, this exponent is the well known critical exponent associated to the $\ell_p$-cohomology of $X$. An analogous exponent $p_{\neq 0}(X,\xi,I)$ can be defined for the local Orlicz cohomology.

\medskip
\noindent Finer invariants can be given following the ideas of \cite{MarcBourdon07,Bourdon-Kleiner13,Bourdon-Kleiner12}. Consider the algebra of continuous Orlicz-Besov functions
$$A^\phi=A^\phi\left(\partial X,\varrho_w\right):=\left\{u\in B^\phi\left(\partial X,\varrho_w\right): u\text{ is continuous}\right\}.$$
It is a unital Banach algebra when equipped with the norm $N_\phi(u):=\|u\|_{\infty}+\langle u\rangle_{\phi}$. The spectrum of $A^{\phi}$, denoted by $\mathrm{Sp}\left(A^{\phi}\right)$, is a Hausdorff compact topological space invariant by Banach algebra isomorphisms. In particular, the spectrum, as well as its topological dimension $s(\phi):=\dim_T\mathrm{Sp}\left(A_{\phi}\right)$, are quasi-isometry invariants of $X$. For instance, given an indexed family $\{\phi_i:i\in I\}$ as before, the function $s:I\to\N\cup\{\infty\},\ i\mapsto s(i)=s(\phi_i),$ provides another quasi-isometry invariant of $X$.

\medskip
\noindent The spectrum of $A_\phi$ is a quotient space of $\partial X$, where the equivalence relation is given by
$$\xi\sim_{\phi}\zeta \text{ if, and only if, } u(\xi)=u(\zeta),\ \forall\,u\in A_\phi.$$
The $\ell^\phi$-equivalence classes provide a partition of $\partial X$ which must be preserved by the boundary homeomorphism of any self quasi-isometry of $X$.

\medskip
\noindent The same considerations can be carried out for the local Orlicz cohomology, by considering the unital Fr\'echet algebra $A^\phi_{\mathrm{loc}}\left(\partial_\xi X,\varrho_\xi\right)$ of continuous Orlicz-Besov functions on the parabolic boundary $\partial_\xi X$. In particular, the boundary homeomorphism of a self quasi-isometry of $X$ which fixes the point $\xi$, must preserve the local $\ell^\phi$-cohomology classes of $\partial_\xi X$.

\subsection{On quasi-isometries of Heintze groups}\label{resultsHeintze}

We will focus on the global and local Orlicz cohomology in degree one of $X_\alpha$, for the family of doubling Young functions given by
\begin{equation}\label{ourphi}
\phi_{p,\kappa}(t)=\frac{|t|^p}{\log\left(e+|t|^{-1}\right)^{\kappa}},\ (p,\kappa)\in I=[1,+\infty)\times[0,+\infty).
\end{equation}
In order to simplify the notation, we indicate the Orlicz spaces and the norms associated to the functions $\phi_{p,\kappa}$ with the superscript ``$p,\kappa$''. The set of pairs $(p,\kappa)\in I$ is endowed with the lexicographic order, so we obtain a non-decreasing family of Young functions as in the previous section.

\medskip
\noindent We will define a non-decreasing sequence of closed Lie subgroups of $N$,
$$\{e\}=K_0\leq H_1\leq K_1\leq H_2\leq K_2\leq \cdots H_d\leq K_d=N,$$ 
whose left cosets will be identified with local cohomology classes for appropriate choices of the parameters $p$ and $\kappa$.

\medskip
\noindent Denote by $\mu_1\leq\cdots\leq\mu_d$ the distinct eigenvalues of $\alpha$, and let $\mathcal{B}$ be a basis of $\mathfrak{n}$ on which $\alpha$ assumes its Jordan canonical form. In the basis $\mathcal{B}$, we have the decomposition
\begin{equation}\label{jordanform}
\alpha=\bigoplus_{i=1}^d\bigoplus_{j=1}^{k_i}J(\mu_i,m_{ij}).
\end{equation}
Here $J(\mu_i,m_{ij})$ denotes a Jordan block of size $m_{ij}$ associated to the eigenvalue $\mu_i$. Let $V_i$ be the generalized eigenspace associated with $\mu_i$.

\medskip
\noindent Let $\mathfrak{k}_0=W_0=\{0\}$, and for each $i\in \{1,\ldots,d\}$, let 
$$W_i=\bigoplus_{r=1}^iV_r,\text{ and } \mathfrak{k}_i=\mathrm{LieSpan}\left(W_i\right).$$
That is, $\mathfrak{k}_i$ is the Lie sub-algebra of $\mathfrak{n}$ generated by $W_i$. Define $K_i$ to be the closed Lie subgroup of $N$ whose Lie algebra is $\mathfrak{k}_i$. Note that the set of left cosets $N/K_i$ is a smooth manifold, and the canonical projection is a smooth map.

\medskip
\noindent For each $i\in\{1,\ldots,d\}$, we also let $m_i=\max\{m_{ij}:1\leq j\leq k_i\}$ be the maximal size of the Jordan blocks with eigenvalue $\mu_i$. Consider the $\mu_i$-eigenvectors belonging to the Jordan blocks of size $m_i$, and denote by $V_i^0$ the vector space their span. Finally, consider $\mathfrak{h}_i\leq \mathfrak{k}_i$ the Lie subalgebra spanned by $W_{i-1}\oplus V_i^0$, and $H_i$ its corresponding closed Lie subgroup. Notice that $\mathfrak{h}_i=\mathfrak{k}_i$ only when $m_i=1$.

\medskip
\noindent The parabolic boundary $\partial_\infty X_\alpha$ can be identified with $N$. A left invariant parabolic visual metric $\varrho_\infty$ can be defined on $N$, and so that $\tau$ acts as a dilation. From this, one easily checks that $\varrho_\infty$ is Ahlfors regular (see Section \ref{hnc} for more details). In the statement of the next theorem we write $A^{p,\kappa}_{\mathrm{loc}}$ for $A_{\mathrm{loc}}^{p,\kappa}\left(N,\varrho_\infty\right)$

\begin{theorem}[Local classes with respect to $\infty$\label{localclassesinfty}]
For each $i\in\{1,\ldots,d\}$, consider the exponent $p_i=\mathrm{tr}(\alpha)/\mu_i$, and set $p_{d+1}=1$.
\begin{enumerate}
\medskip
\item Suppose $p\in \left(p_{i+1},p_i\right)$ and $\kappa\geq 0$. The spectrum of $A^{p,\kappa}_{\mathrm{loc}}$ is homeomorphic to $N/K_i$. 

\medskip
\item Suppose $p=p_i$ and $m_i=1$.

\begin{enumerate}
\medskip
\item If $\kappa\leq 1$, the spectrum of $A^{p,\kappa}_{\mathrm{loc}}$ is homeomorphic to $N/K_i$.

\medskip
\item If $\kappa>1$, the spectrum of $A^{p,\kappa}_{\mathrm{loc}}$ is homeomorphic to $N/K_{i-1}$.
\end{enumerate}

\medskip
\item Suppose $p=p_i$, $\kappa\in \left(1+p_i(m_i-2),1+p_i(m_i-1)\right]$, and $m_i\geq 2$ . The spectrum of $A^{p,\kappa}_{\mathrm{loc}}$ is homeomorphic to $N/H_i$.
\end{enumerate}

\medskip
\noindent In particular, in all cases, the $(p,\kappa)$-local cohomology classes on $N$ coincide with the lefts cosets of the corresponding subgroup.
\end{theorem}

\noindent Let $i(\alpha)=\min\{i:K_i=N\}$. Then the local critical exponent satisfies 
$$(p_{i(\alpha)},0)\leq p_{\neq 0}\left(X_\alpha,\infty,I\right)\leq\left(p_{i(\alpha)},1+p_{i(\alpha)}\left(m_{i(\alpha)}-2\right)^+\right),$$
where $a^+=\max\{a,0\}$. 

\medskip
\noindent The picture is quite different for the local cohomology with respect to the points of $N$.

\begin{theorem}[Local cohomology with respect to $\xi\in N$\label{localclassesN}] The local $\ell^{\,p,\kappa}$-cohomology  of $X_\alpha$ with respect to a point $\xi\in N$ is trivial if, and only if, $(p,\kappa)\leq \left(p_1,1+p_1(m_1-1)\right)$.
\end{theorem}

\noindent In particular, $p_{\neq 0}(X_\alpha,\xi,I)=\left(p_1,1+p_1(m_1-1)\right)$ for any $\xi\in N$. Theorem \ref{fixpoint} follows therefore from Theorems \ref{localclassesinfty} and \ref{localclassesN}, see Section \ref{hnc}. Notice that in the Carnot type case, $p_{\neq 0}(X_\alpha,\infty,I)=p_{\neq 0}(X_\alpha,\xi,I)=(p_1,1)$ for any $\xi\in N$. As an immediate consequence, we obtain the following result for the global cohomology.

\begin{corollary}\label{globalheintze}
The critical exponent of the $\ell^{\,p,\kappa}$-cohomology of $X_\alpha$ is given by 
$$p_{\neq 0}(X_\alpha,I)=\left(p_1,1+p_1(m_1-1)\right).$$
Moreover, the $\ell^{\,p,\kappa}$-cohomology is also trivial at this critical exponent.
\end{corollary}

\noindent As an example, consider the Heintze group $X_3$ introduced at the beginning of this section. The critical exponents are in this case is $p_{\neq 0}(X_3,\infty,I)=(2,3)$ and $p_{\neq 0}(X_3,\xi,I)=(2,1)$ for any $\xi\in\R^2$. The subgroup $H_1$ is $\R\times \{0\}$. Notice that even though the conformal dimension of $\partial X_3$ is equal to $2$, it is not attained. This is proved in \cite{HaissinskyPilgrim11}[Thm.\,1.\,8] by techniques of two dimensional conformal dynamics which do not apply to the higher dimensional case. We refer the reader to \cite{MackayTyson10} for an account on conformal dimension.

\medskip
\noindent The $\kappa$-coordinate of the critical exponent can be interpreted as a second order quasi-isometry invariant ``dimension'' of a hyperbolic complex $X$ as above. When the Ahlfors regular conformal dimension $Q$ of $\partial X$ is attained, the critical exponents satisfy the inequality $p_{\neq 0}(X,\xi,I)\leq(Q,1)$ for any $\xi\in\partial X$, see Lemma \ref{localLip}.

\begin{corollary}\label{confdimheintze}
If $m_1\geq 2$, the Ahlfors regular conformal dimension of $\partial X_\alpha$ is not attained.
\end{corollary}

\noindent The pointed sphere conjecture is not settled in the Carnot type case. Our methods do not apply, essentially, because Carnot groups equipped with Carnot-Carath\'eodory metrics are Loewner spaces \cite{Heinonen01}. In particular, they contain ``a lot of rectifiable curves'', which makes difficult to distinguish points by invariants strongly related to the moduli of curves.

\medskip
\noindent The restrictions imposed on self quasi-isometries of $X_\alpha$ preserving a foliation at infinity are manifest in the rigidity results due to Xie, in the case when $N$ is abelian or isomorphic to a Heisemberg group. It is there shown that self quasi-isometries are \emph{almost-isometries}, i.e.\ a $(1,\mathsf{C})$-quasi-isometry. This question is motivated by the work of Farb and Mosher on abelian-by-cyclic groups \cite{FarbMosher00}. We apply Xie's approach to our more general context. 

\begin{corollary}\label{biLip}
Let $X_\alpha$ be a purely real Heintze group, and suppose that the normalizer of $\mathfrak{h}_1$, $N_{\mathfrak{n}}(\mathfrak{h}_1)=\left\{v\in\mathfrak{n}:[v,w]\in\mathfrak{h}_1,\ \forall w\in\mathfrak{h}_1\right\}$, is strictly bigger than $\mathfrak{h}_1$. Then, any self quasi-isometry of $X_\alpha$ is an almost isometry.
\end{corollary}

\noindent The Corollary \ref{biLip} applies, in particular, in the case when $N$ is abelian or isomorphic to a Heinsemberg group, and $X_\alpha$ is not of Carnot type, generalizing therefore the previous known results.

\medskip
\noindent We obtain also results regarding the quasi-isometric classification of Heintze groups. It is conjectured that two purely real Heintze groups are quasi-isometric if, and only if, they are isomorphic \cite{UrsulaHamenstadt87,Cornulier13}. By \cite[Theorem 2]{PierrePansu89m}, the conjecture is true when $X_\alpha$ and $X_\beta$ are both purely real Heintze groups of Carnot type. Notice that if $\alpha_1$ is the restriction of $\alpha$ to $\mathfrak{h}_1$, then $H_1\rtimes_{\alpha_1} \R$ is a Heintze group of Carnot type. Also, that if there exists a quasi-isometry between two Heintze groups $X_\alpha$ and $X_\beta$, then there exists a quasi-isometry sending $\infty$ to $\infty$ \cite[Lemma 6.D.1]{Cornulier13}. From Pansu's theorem and Theorem \ref{localclassesinfty}, we obtain the following consequence.

\begin{corollary}\label{carnotnoncarnot}
Let $X_\alpha$ and $X_\beta$ be two purely real Heintze groups. If they are quasi-isometric, then $H_1^{(\alpha)}\rtimes_{\alpha_1} \R$ and $H_1^{(\beta)}\rtimes_{\beta_1} \R$ are isomorphic. In particular, if $X_\alpha$ is of Carnot type and $X_\beta$ is not, then they are not quasi-isometric.
\end{corollary} 

\noindent In the abelian type case, that is, when $N$ is abelian, the algebra is not playing any role, and we obtain another, and more direct, proof of the following result of Xie.

\begin{corollary}\cite[Theorem 1.1]{XiangdongXie14}\label{abelian}
Let $X_\alpha$ and $X_\beta$ be two purely real Heintze groups of abelian type. If they are quasi-isometric, there exists $\lambda>0$ such that $\alpha$ and $\lambda\beta$ have the same Jordan form. In particular, $X_\alpha$ and $X_\beta$ are isomorphic.
\end{corollary}
\noindent We refer the reader to \cite[Section 6.B]{Cornulier13} for more details about the quasi-isometric classification of Heintze groups. Notice that Theorem \ref{localclassesinfty} provides new invariants related to the sizes of the Jordan blocks of $\alpha$. We were not able to compute the spectrum of $A^{p,\kappa}_{\mathrm{loc}}$ for all the possible values of $(p,\kappa)$. One could expect to show, by a more refined analysis, that the Jordan form of the derivation $\alpha$, up to scalar multiplication, is a quasi-isometry invariant of the Heintze group, generalizing thus the first conclusion of Corollary \ref{abelian}.

\subsection{Notations and conventions} To make the notation clearer we will write the constants in sans-serif font, e.g.\ $\sf{C}$, $\sf{K}$, etc..\ If $f$ and $g$ are non-negative real functions defined on a set $A$, we say that $f\lesssim g$ if there exists a constant $\sf{C}$, such that $f(a)\leq \sf{C} g(a)$ for all $a\in A$. If both inequalities are true, $f\lesssim g$ and $g\lesssim f$, we say that $f$ and $g$ are comparable and we denote it by $f\asymp g$.

\subsection*{Acknowledgments} I would like to specially thanks Pierre Pansu for all his help and advice during the realization of this work, and for explaining me his works on $L^p$-cohomology. I am also particularly indebted to Romain Tessera for sharing with me his ideas about Orlicz spaces and their applications to Heintze groups. I also thanks Yves de Cornulier for helpful discussions on quasi-isometries of Heintze groups. This work was supported by the ANR project ``GDSous/GSG'' no.\ 12-BS01-0003-01.

\section{Preliminaries on the theory of Orlicz spaces}\label{preOrlicz}

\subsection{Generalities} In this section we recall some basic facts about Orlicz spaces which will be used throughout this articles. We refer to \cite{Rao-Ren91} for a general treatment. Recall that a \emph{Young function} is an even convex function $\phi:\R\to\R_+$ which satisfies $\phi(0)=0$, and $\lim_{t\to\infty}\phi(t)=+\infty$. Note that we require $\phi$ to be finite valued, so in particular, $\phi$ is locally Lipschitz. We will also assume that $\phi(t)\neq 0$ if $t\neq 0$.

\medskip
\noindent Any Young function can be represented as an integral
\begin{equation*}\label{derivative}\phi(t)=\int_0^{|t|}\phi'(s)ds,\end{equation*}
where $\phi':\R_+\to \R_+$ is nondecreasing left continuous and $\phi'(0)=0$ \cite[Cor.\ 2 Ch.\ 1]{Rao-Ren91}. The function $\phi'$ coincides with the derivative of $\phi$ except perhaps for at most a countable number of points.

\medskip
\noindent Let the space $\Omega$ be given with a $\sigma$-algebra and a $\sigma$-finite measure $\mu$. For any measurable function $f$ on $\Omega$, the \emph{Luxembourg} norm of $f$ is defined as
$$\|f\|_\phi=\inf\left\{\alpha>0:\int_{\Omega}\phi\left(\frac{f}{\alpha}\right)d\mu\leq 1\right\}\in[0,+\infty],$$
where it is understood that $\inf(\emptyset)=+\infty$. The \emph{Orlicz space} $L^\phi(\Omega,\mu)$ is the vector space of measurable functions $f$ on $\Omega$ such that $\|f\|_{\phi}<\infty$ \cite[Thm.\ 10 Ch.\ 3]{Rao-Ren91}. Up to almost everywhere null functions, $L^\phi(\Omega,\mu)$ is a Banach space with the norm $\|\cdot\|_\phi$. As usual, we simply write $\ell^\phi(\Omega)$ when $\mu$ is the counting measure.

\medskip
\noindent If $\sf{K}\geq 1$ is any constant, the identity map $L^{\sf{K}\phi}(\Omega,\mu)\to L^\phi(\Omega,\mu)$ is continuous and bijective, and therefore, by the open mapping theorem, $\|\cdot\|_{\sf{K}\phi}$ and $\|\cdot\|_\phi$ are equivalent norms:
\begin{equation}\label{equiv}\exists\, \sf{C}=\sf{C}(\sf{K},\phi)\text{ s.t. } \|\cdot\|_\phi\leq \|\cdot\|_{\sf{K}\phi}\leq \sf{C}\|\cdot\|_{\phi}.\end{equation}
Let us point out that when the measure $\mu$ is finite, the Orlicz space $L^\phi(\Omega,\mu)$ is contained in  $L^1(\Omega,\mu)$ and the inclusion is continuous \cite[Cor.\ 3 Ch.\ 1]{Rao-Ren91}. In particular, in locally compact spaces equipped with a regular measure, Orlicz integrable functions are locally integrable.

\medskip
\noindent We will not need the H\"older inequality in this paper, but let us remark that it holds exactly as in the ordinary case by considering $\psi:\R\to [0,+\infty]$ the convex conjugate of $\phi$ \cite[Thm.\,3 Ch.\, 1]{Rao-Ren91}. This inequality serves also to define the Orlicz spaces by means of the so called Orlicz norm instead of the Luxembourg norm.

\medskip
\noindent The main difficulty in dealing with the norm $\|\cdot\|_{\phi}$ is that it is not comparable to the function
\begin{equation}\label{difficulty}
f\mapsto \phi^{-1}\left(\int_\Omega\phi(f)du\right).
\end{equation}
This is only the case when $\phi$ is equivalent to an ordinary power function $\phi_p$, $p\geq 1$. Nevertheless, we will be able to avoid this difficulty by applying Jensen's inequality in order to exchange $\phi$ and an integral or a sum symbol.

\subsection{Decay conditions}
For the Orlicz spaces considered here, only the decay properties of $\phi$ for small values of $t$ will be relevant for us. Roughly speaking, we do not want our Young functions to be too small near zero. This is precisely the meaning of Definition \ref{growth}, and we will mainly work in this article with doubling Young functions.

\medskip
\noindent Here are some important features of Orlicz spaces of doubling Young functions. First, notice that since $\phi'$ is non-decreasing, we have
$$\frac{t\phi'(t)}{\phi(t)}\geq 1\text{ for all } t>0.$$
A useful way to check the doubling property is by considering the exponent
$$p_{\phi}=\limsup_{t\to 0}\frac{t\phi'(t)}{\phi(t)}\in [1,+\infty].$$
Then, $\phi$ is doubling if and only if $p_{\phi}<+\infty$ \cite[Cor.\ 4 Ch.\ 2]{Rao-Ren91}. It is important to note that when $\phi$ is doubling, then 
$$f\in L^\phi(\Omega,\mu) \text{ if, and only if, }\int_\Omega\phi(f)d\mu<\infty,$$
see \cite[Thm.\ 2 Ch.\ 3]{Rao-Ren91}. Moreover, when $\phi$ is doubling, if $\{f,f_n: n\in \N\}$ is a sequence in $L^\phi(\Omega,\mu)$, then
$$\|f_n-f\|_\phi\to 0 \text{ if, and only if, } \int_\Omega \phi(f_n-f)du\to 0,$$
see \cite[Thm.\,12, Ch.\,3]{Rao-Ren91}. In particular, when $\mu$ is a Radon (resp. volume) measure on a locally compact Hausdorff space (resp. Riemannian manifold) $\Omega$, then the usual approximation by continuous (resp. smooth) functions in the $\|\cdot\|_\phi$ norm holds.

\medskip
\noindent A doubling Young function has polynomial decay with exponent $p$ for any $p>p_{\phi}$. That is, there exist $t_0>0$, and $\sf{c}>0$, such that $\phi(t)\geq \sf{c}\,t^p \text{ for } t\in[0,t_0]$. In general, the converse is not true. We also remark that when $\phi$ is doubling of constant $\sf{K}$, then $t\phi'(t)\leq \sf{K}\phi(t)$ for all $t\in [0,t_0]$.

\medskip
\noindent All these conditions are imposed for small values of $t$. When dealing with discrete Orlicz spaces, the behavior of $\phi$ at infinity is essentially irrelevant. To justify this, let us introduce the following equivalence relation among Young functions.

\medskip
\noindent Define $\phi_1 \preceq \phi_2$ if there are constants $\sf{a},\sf{b}\in\R_+$ and $t_0>0$, so that $\phi_2(t)\leq \sf{a}\phi_1(\sf{b}t)$ for $t\leq t_0$. We say that $\phi_1\sim \phi_2$ for small $t$, if $\phi_1\preceq \phi_2$ and $\phi_2\preceq \phi_1$.

\begin{lemma}\label{equivalentphi}
Suppose that $\Omega$ is a countable set, and let $\phi_i$, $i=1,2$, be a pair of equivalent Young functions for small $t$. Then the norms $\|\cdot\|_{\phi_1}$ and $\|\cdot\|_{\phi_2}$ are equivalent.
\end{lemma}

\noindent In particular, given a doubling Young function $\phi$, applying Lemma \ref{equivalentphi} if necessary, we can change $\phi$ by an equivalent (for small $t$) Young function, so that the doubling condition \ref{growth} is satisfied with $t_0=+\infty$.

\medskip
\noindent The following lemma will be used later in Section \ref{boundary}.

\begin{lemma}\label{lowerdoubling}
Let $\phi$ be a doubling Young function with constants $\sf{K}$ and $t_0=+\infty$. Then for all $y\in \R_+$ and $x\in[0,1]$ we have $x^\sf{K}\phi(y)\leq \phi(xy)$.
\end{lemma}
\begin{proof}
For each $x\in [0,1]$, define the function $\tau_x:\R_+\to\R_+$ by $\tau_x(y)=\phi(xy)/\phi(y)$. Notice that $\tau_x(y)\leq 1$ because $x\leq 1$ and $\phi$ is increasing. Since $t\phi'(t)\leq \phi(t)$ for all $t\in\R_+$, we have
$$\log\left(\frac{1}{\tau_x(y)}\right)=\int_{xy}^y\frac{\phi'(t)}{\phi(t)}dt\leq \int_{xy}^y\frac{\sf{K}}{t}dt=\log\left(\frac{1}{x^\sf{K}}\right).$$
That is, $\tau_x(y)\geq x^\sf{K}$. This finishes the proof.
\end{proof}

\section{Orlicz cohomology and quasi-isometry invariance\label{qiinvariance}}

\noindent In this section, we extend the classical notion of $L^p$-cohomology by considering the larger class of Orlicz spaces. It can be defined as a simplicial, coarse, or De Rham cohomology, depending on the structure of the space under consideration. For our later applications in Section \ref{hnc}, it will be more natural to work with differential forms. Nevertheless, the quasi-isometry invariance is easier to prove in the simplicial context.

\medskip
\noindent In contrast to the $L^p$ case, the identification of the discrete and continuous cohomologies is less clear for a general Young function. We will prove it for the degree one cohomology, by using the coarse definition as an intermediary (see Section \ref{coarsec}).

\subsection{The simplicial Orlicz chomology}\label{simplicial}

\noindent We first prove the quasi-isometry invariance of the simplicial Orlicz cohomology. It is remarkable that this result is true for any Young function $\phi$, and therefore, in great generality. The proof follows the same lines as the classical proof for the $\ell^p$ case, see \cite{Bourdon-Pajot04}.
 
\medskip
\noindent Let $X$ be a finite dimensional, uniformly contractible, simplicial complex with bounded geometry, see the Introduction for the definition. We denote the geodesic distance on $X$ by $|\cdot-\cdot|$. For $k\in \N$, let $X_k$ be the set of $k$-simplices of $X$, and let $C_k(X)$ be the vector space of $k$-chains on $X$, i.e.\ finite real linear combinations of the elements in $X_k$.

\medskip
\noindent We will use the following notations: if $c$ is a chain in $X$, the length of $c$, denoted by $\ell(c)$, is the number of simplices in $\supp(c)$, $\|c\|_\infty=\max |c(\sigma)|$, and $\|c\|_1=\sum|c(\sigma)|$. Note that our assumptions imply that there exists a function $\sf{N}_X:[0,+\infty)\to\N$ such that any ball of radius $r\geq 0$ contains at most $\sf{N}_X(r)$ simplices. We write $\sf{N}=\sf{N}_X(1)$ for short.

\medskip
\noindent Let $\phi$ be a Young function. The \emph{$k$-th space of $\phi$-integrable cochains} of $X$ is by definition the Banach space $\Lp^\phi(X_k)$; i.e.\ the space of functions $\omega:X_k\to \R$ such that
$$\|\omega\|_\phi:=\inf\left\{\alpha>0:\sum_{\sigma\in X_k}\phi\left(\frac{\omega(\sigma)}{\alpha}\right)\leq 1\right\}<\infty.$$
The standard coboundary operator $\delta_k:\Lp^\phi(X_k)\to \Lp^\phi(X_{k+1})$ is defined by duality: for $\omega\in \Lp^\phi(X_k)$ and $\sigma\in X_{k+1}$, let $\delta_k(\omega)(\sigma)=\omega(\partial \sigma)$.

\medskip
\noindent By the convexity of $\phi$, and the bounded geometry assumption, $\delta_k$ is well defined and is a bounded operator. The proof relies on the following argument which will be used several times in this article. Let us explain it in detail.

\medskip
\noindent Let $\omega\in \Lp^\phi(X_k)$ and $\sigma\in X_{k+1}$, and suppose that $\partial\sigma\neq 0$. Note that for any $\alpha>0$, we have by convexity
\begin{align}\label{keyarg}
\phi\left(\frac{\delta_k(\omega)(\sigma)}{\alpha \sf{N}}\right)&\leq \phi\left(\frac{\delta_k(\omega)(\sigma)}{\alpha \|\partial\sigma\|_1}\right)
\leq\frac{1}{\|\partial\sigma\|_1}\sum_{\sigma'\in \supp(\partial\sigma)}\phi\left(\frac{\omega(\sigma')}{\alpha}\right).
\end{align}
Summing over $\sigma\in X_{k+1}$, we obtain
\begin{align*}
\sum_{\sigma\in X_{k+1}}\phi\left(\frac{\delta_k(\omega)(\sigma)}{\alpha \sf{N}}\right)
&\leq \sf{N}\sum_{\sigma'\in X_k}\phi\left(\frac{\omega(\sigma')}{\alpha}\right).
\end{align*}
By (\ref{equiv}), $\|\omega\|_{\sf{N}\phi}\lesssim \|\omega\|_\phi<\infty$. In particular, the set of $\alpha>0$ such that the sum on the right hand side is bounded from above by $\sf{N}^{-1}$ is not empty, and we can take $\alpha$ arbitrarily close to $\|\omega\|_{\sf{N}\phi}$. For such an $\alpha$, we have
\begin{align}\label{keyarg2}
\sum_{\sigma\in X_{k+1}}\phi\left(\frac{\delta_k(\omega)(\sigma)}{\alpha \sf{N}}\right)&\leq 1,
\end{align}
which implies $\|\delta_k(\omega)\|_\phi\leq \alpha$. Taking infimum, we get $\|\delta_k(\omega)\|_\phi\leq \sf{N}\|\omega\|_{\sf{N}\phi}$. Therefore, $\|\delta_k(\omega)\|_\phi\lesssim \|\omega\|_\phi$, where the multiplicative constant depends only on $\sf{N}$ and $\phi$.

\begin{definition}\label{cohomology}
The \emph{$k$-th $\Lp^\phi$-cohomology space} of $X$ is by definition the topological vector space
\begin{align*}\Lp^\phi H^k(X):=\mathrm{Ker}\,\delta_k/\mathrm{Im}\,\delta_{k-1}.\end{align*}
The \emph{reduced} cohomology is defined by taking the quotient by $\overline{\mathrm{Im}\,\delta_{k-1}}$; it is a Banach space.
\end{definition}

\noindent We focus now on the quasi-isometry invariance of the $\Lp^\phi$-cohomology spaces. Recall that a function $F:X\to Y$ between two metric spaces is:
\begin{enumerate} 
\item \emph{quasi-Lipschitz} if there exist constants $\sf{\Lambda}\geq 1$ and $\sf{C}\geq0$ such that
\begin{align*}
\forall\ x,y\in X,\ |F(x)-F(y)|\leq \sf{\Lambda} |x-y| +\sf{C}.
\end{align*}
\item \emph{uniformly proper} if there exists a function $\sf{D}_F:\R_+\to\R_+$ such that for any ball $B(y,r)$ in $Y$, we have $\diam\left(F^{-1}(B(y,r))\right)\leq \sf{D}_F(r)$.
\medskip
\item a \emph{quasi-isometry} if it is quasi-Lipschitz and there exists a quasi-Lipschitz function $G:Y\to X$ such that $G\circ F$ and $F\circ G$ are at bounded distance from the identity.
\end{enumerate}

\noindent We can now state the main result of this section.

\begin{theorem}[Quasi-isometry invariance]\label{invariance}
Let $\phi$ be any Young function, and $X$, $Y$ be two uniformly contractible simplicial complexes with bounded geometry.
\begin{enumerate}
\item Any quasi-Lipschitz uniformly proper function $F:X\to Y$ induces continuous linear maps $F^*:\Lp^\phi H^\bullet(Y)\to \Lp^\phi H^\bullet(X)$.
\item If $F,G:X\to Y$ are two quasi-Lipschitz uniformly proper functions at bounded distance, then $F^*=G^*$.
\item If $F:X\to Y$ is a quasi-isometry, then $F^*$ is an isomorphism of topological vector spaces.
\end{enumerate}
The same statements hold for the reduced $\Lp^\phi$-cohomology.
\end{theorem}

\noindent The proof relies on the following two key lemmas which are proved in \cite[Section 1]{Bourdon-Pajot04}. They serve to define pull-backs and homotopies from quasi-isometries.

\begin{lemma}[\cite{Bourdon-Pajot04}]\label{pullback}
Let $X$ and $Y$ be two uniformly contractible simplicial complexes with bounded geometry. Any quasi-Lipschitz uniformly proper function $F:X\to Y$ induces maps $c_F:X_\bullet\to C_\bullet(Y)$ verifying the following conditions:
\begin{enumerate}
\item $c_F$ commutes with the boundary operator, i.e. $c_F(\partial \sigma)=\partial c_F(\sigma)$, and 
\item for each $k\in \N$, there are constants $\sf{N}_k$ and $\sf{L}_k$, depending only on $k$ and the geometric data of $X,\ Y$ and $F$, such that
$$\|c_F(\sigma)\|_\infty\leq \sf{N}_k,\ \text{and}\ \ell(c_F(\sigma))\leq \sf{L}_k.$$
\end{enumerate}
\end{lemma}

\noindent The map $c_F:X_\bullet\to C_\bullet(Y)$ is constructed by induction on $k$. For $k=0$ and $\sigma\in X_0$,  $c_F(\sigma)$ is defined to be any vertex of $Y$ at uniformly bounded distance from $F(\sigma)$. The induction can be carried out since $X$ is uniformly contractible.

\begin{lemma}[\cite{Bourdon-Pajot04}]\label{homotopy}
Let $F,G:X\to Y$ be two quasi-Lipschitz uniformly proper functions at bounded uniform distance. Then there exists a homotopy $h:X_\bullet\to C_{\bullet+1}(Y)$ between $c_F$ and $c_G$. That is, a map verifying
\begin{enumerate}
\item[(i)] for $\sigma\in X_0$, $\partial h(\sigma)=c_F(\sigma)-c_G(\sigma)$, and
\item[(ii)] for $\sigma\in X_k$, $k\geq 1$, $\partial h(\sigma)+h(\partial\sigma)=c_F(\sigma)-c_G(\sigma)$.
\end{enumerate}
As before, $\|h(\sigma)\|_\infty$ and $\ell(h(\sigma))$ are uniformly bounded on $X_k$ by constants $\sf{N}'_k$ and $\sf{L}_k'$, depending only on the geometric data of $X,\ Y,\ F$ and $G$.
\end{lemma}

\noindent The homotopy $h$ is also defined by induction on $k$. For $k=0$ and $\sigma\in X_0$, $h(\sigma)$ is defined to be any $1$-chain in $Y$ verifying
$$\partial h(\sigma)=c_F(\sigma)-c_G(\sigma),\ \ell(h(\sigma))\lesssim |c_F(\sigma)-c_G(\sigma)|, \text{ and }\|h(\sigma)\|_\infty=1.$$

\begin{proof}[Proof of Theorem \ref{invariance}]
Let us prove (1). For $\omega\in \Lp^\phi(Y_k)$, define the pull-back 
$$F^*(\omega):X_k\to \R,\ F^*(\omega)(\sigma)=\omega\left(c_F(\sigma)\right),\ \sigma\in X_k.$$
We show that $F^*(\omega)$ belongs to $\Lp^\phi(X_k)$. The proof will also show the continuity of $F^*$. For $\alpha>0$, by convexity of $\phi$, we have
\begin{align*}
\sum_{\sigma\in X_k}\phi\left(\frac{F^*(\omega)(\sigma)}{\sf{N}_k\sf{L}_k\alpha}\right)&\leq \sum_{\sigma\in X_k}\sum_{\sigma'\in\supp\,c_F(\sigma)}\phi\left(\frac{\omega(\sigma')}{\alpha}\right).
\end{align*}
For fixed $\sigma'\in Y_k$, the uniform properness of $F$ implies that the set of $\sigma\in X_k$ for which $\sigma'\in \supp\ c_F(\sigma)$, is contained in a ball of $X$ of radius $r=\sf{D}_F(2\sf{L}_k+\sf{\Lambda}+\sf{C})$. Then, their number is bounded by $\sf{N}_X(r)$. This implies
$$\sum_{\sigma\in X_k}\sum_{\sigma'\in\supp\,c_F(\sigma)}\phi\left(\frac{\omega(\sigma')}{\alpha}\right)\leq \sf{N}_X(r)\sum_{\sigma'\in Y_k}\phi\left(\frac{\omega(\sigma')}{\alpha}\right).$$
As in (\ref{keyarg2}), this gives $\|F^*(\omega)\|_\phi\leq \sf{N}_k\sf{L}_k\|\omega\|_{\left(\sf{N}_X(r)\phi\right)}\lesssim\|\omega\|_\phi.$

\medskip
\noindent Denote by $\delta_k^X$ and $\delta_k^Y$ the coboundary operators of $X$ and $Y$ respectively. Then $\delta_k^X\circ F^*=F^*\circ \delta_k^Y$, since $c_F$ commutes with the boundary operator. Thus, $F^*$ induces a continuous linear map $F^*:\Lp^\phi H^k(Y)\to \Lp^\phi H^k(X)$. This also holds for the reduced cohomology by the continuity of $F^*$.

\medskip
\noindent Let us prove (2). Suppose that $\|F-G\|_\infty<\infty$. For $\omega\in \Lp^\phi(Y_{k+1})$, define the pull-back 
$$h^*(\omega):X_k\to \R,\ h^*(\omega)(\sigma)=\omega\left(h(\sigma)\right),\ \sigma\in X_k.$$
As in the proof of (1), by the convexity of $\phi$, we have
$$\sum_{\sigma\in X_k}\phi\left(\frac{h^*(\omega)(\sigma)}{\sf{N}'_k\sf{L}'_k\alpha}\right)\leq\sum_{\sigma\in X_k}\sum_{\sigma'\in\supp\,h(\sigma)}\phi\left(\frac{\omega(\sigma')}{\alpha}\right).$$
For fixed $\sigma'\in Y_{k+1}$, the set of $\sigma\in X_k$ for which $\sigma'\in \supp\ h(\sigma)$, is contained in a ball of radius $r'$ which depends only on $\sf{L}_k'$ and the geometric data of $X,Y,F$ and $G$. Therefore, as in (1), we obtain $\|h^*(\omega)\|_\phi\leq \sf{N}'_k\sf{L}'_k\|\omega\|_{(\sf{N}_X(r')\phi)}\lesssim\|\omega\|_{\phi}$. This shows that $h^*:\Lp^\phi(Y_{k+1})\to \Lp^\phi(X_k)$ is a bounded operator. By Lemma \ref{homotopy}, we have
$$h^*\circ \delta_k^Y+\delta_{k-1}^X\circ H^*=F^*-G^*,$$
which implies that $F^*$ and $G^*$ induce the same map in $\Lp^\phi$-cohomology.

\medskip
\noindent Finally, (3) results from (2) since $F^*\circ G^*$ and $G^*\circ F^*$ are the identity maps.
\end{proof}

\subsection{The degree one coarse Orlicz cohomology.}\label{coarsec}

Let $(X,d)$ be a complete proper metric space equipped with a Radon measure $\mu$. We say that $X$ has \emph{bounded geometry} if
\begin{equation}\label{bgeom}
0<v(r)=\inf_{x\in X}\left\{\mu\left(B(x,r)\right)\right\}\leq V(r)=\sup_{x\in X}\left\{\mu\left(B(x,r)\right)\right\}<\infty,\, \forall\, r>0.
\end{equation}
We say that $X$ has the \emph{midpoint property} if there exists a constant $\sf{c}_X\geq 0$ such that for all $x,y\in X$, there is some $z\in X$ with
$$\max\left\{d(z,x),d(z,y)\right\}\leq \frac{d(x,y)}{2}+\sf{c}_X.$$
This property is satisfied for instance, if $X$ is $(1,\sf{c})$-quasi-isometric to a geodesic metric space, with a mid-point constant equal to $3\sf{c}$. In the rest of this section we suppose that $X$ has bounded geometry and satisfies the mid-point property. This and the next section are inspired by \cite{PierrePansu89,VladimirShchur14} and we will closely follow the ideas there exposed.

\medskip
\noindent A \emph{kernel} in $X$ is a bounded function $\kappa:X\times X\to\R_+$ such that
\begin{enumerate}
\item for all $x\in X$, 
$$\int_X\kappa(x,y)d\mu(y)=1,\text{ and }$$
\item there are constants $\varepsilon>5\sf{c}_X$, $\delta>0$ and $R>0$ such that
$$\kappa(x,y)\geq\delta\text{ if }d(x,y)\leq\varepsilon\text{ and }\kappa(x,y)=0\text{ if }d(x,y)\geq R.$$
\end{enumerate}
Such a kernel always exists, consider for example
\begin{equation}\label{kernel}
\kappa(x,y)=\frac{1}{\mu\left(B(x,r)\right)}\I_{\{d(x,y)<r\}}\ (\text{with }r >5\sf{c}_X).
\end{equation}
The positivity radius of a kernel $\kappa$ is the best constant $\varepsilon_\kappa$ for which (2) is verified. The convolution of two kernels $\kappa_1$ and $\kappa_2$ is the kernel given by
$$\kappa_1*\kappa_2(x,y)=\int_X\kappa_1(x,z)\kappa_2(z,y)d\mu(z).$$
Notice that by the definition of $\sf{c}_X$, the positivity radius of self convolutions of a kernel $\kappa$ satisfies
$$\varepsilon_{\kappa^{*2}}\geq \frac{6}{5}\,\varepsilon_{\kappa},\text{ and in particular, } \varepsilon_{\kappa^{*2^m}}\to +\infty\text{ when }m\to+\infty.$$
This implies that given any two kernels $\kappa_1$ and $\kappa_2$ on $X$, there is some $m\in\N$ and some constant $\sf{C}$ (depending only on the kernels) such that 
\begin{equation}\label{dominated}
\kappa_1\leq \sf{C}\kappa_2^{*2^m}.
\end{equation}
Given a kernel $\kappa$, we denote by $\mu_\kappa$ the measure $\kappa\,d\mu\otimes d\mu$ on $X^2=X\times X$. The space of $(\phi,\kappa)$-integrable cocycles on $X$ is the closed subspace of $L^\phi\left(X^2,\mu_\kappa\right)$ given by the measurable functions $\omega:X\times X\to \R$ verifying $\omega(x,y)=\omega(x,z)+\omega(z,y)$ for almost every $x,y,z\in X$. We denote by $Z_{\phi,\kappa}(X)$ the Banach space of $(\phi,\kappa)$-integrable cocycles on $X$. We write $\|\cdot\|_{\phi,\kappa}$ for the corresponding Luxembourg norm of a cocycle.

\medskip
\noindent Given a measurable function $u:X\to \R$, we can define the cocycle $Du(x,y)=u(x)-u(y)$. We obtain in this way a bounded operator $D:L^\phi(X,\mu)\to Z_{\phi,\kappa}(X)$. We define the coarse $L^\phi_\kappa$-cohomology in degree one as the quotient $L^\phi_\kappa H^1(X):=Z_{\phi,\kappa}(X)/\rm{Im }(D)$.

\medskip
\noindent Observe that by (\ref{dominated}), different kernels define equivalent $\phi$-norms when restricted to cocycles. That is, the integrability condition is independent of the kernel. Notice also that a measurable function $u:X\to\R$ with $\|Du\|_{\phi,\kappa}<\infty$ is in $L^1_{\mathrm{loc}}(X)$.

\medskip
\noindent We need to define the convolution of functions and cocylces with respect to a kernel $\kappa$. Given a measurable function $u:X\to \R$, we define the convolution
\begin{equation}\label{convfunc}
u*\kappa(x)=\int_X u(z)\kappa(x,z)d\mu(z),
\end{equation}
and given a cocycle $\omega:X\times X\to \R$, we define the convolution
\begin{equation}\label{convcocycle}
\omega*\kappa(x,y)=\int_{X^2}\omega(z,z')\kappa(x,z)\kappa(y,z')d\mu(z)d\mu(z').
\end{equation}
Notice the following two important facts about convolutions. First, that it commutes with the coarse coboundary operator: $D(u*\kappa)=(Du)*\kappa$. And second, that
$$\omega*\kappa-\omega=Du,\text{ where }u(x)=\int_X\omega(x,z)\kappa(x,z)d\mu(z).$$
This implies that $\omega*\kappa$ is a cocycle, and by Jensen's inequality, that $\|u\|_\phi\leq \|\omega\|_{\phi,\kappa}$. This also shows that the convolution induces the identity map in cohomology.

\begin{proposition}\label{qiinvcoarse}
Let $X$ and $Y$ be two complete proper metric spaces equipped with Radon measures $\mu_X$ and $\mu_Y$ satisfying the conditions (\ref{bgeom}) and satisfying the midpoint property. Let $F:X\to Y$ be a $(\sf{\Lambda},\sf{c})$-quasi-isometry between them. Then there are kernels $\kappa$, $\kappa_Y$ in $Y$ and $\kappa_X$ in $X$ such that the map
$$F^*:\omega\mapsto \left(\omega*\kappa\right)\circ F$$
induces an isomorphism from $L^\phi_{\kappa_Y}H^1(Y)$ to $L^\phi_{\kappa_X}H^1(X)$.
\end{proposition}
\begin{proof}
Choose $R>\max\{\sf{\Lambda}+2\sf{c},5\sf{c}_Y\}$, and consider the kernel $\kappa$ on $Y$ given by (\ref{kernel}) with $r=R$. Note that for any $z\in Y$, we have
\begin{equation}\label{est1}
\frac{v_X(1)}{V_Y(R)}\leq I(z)\leq \frac{V_X(\sf{\Lambda}(R+2\sf{c}))}{v_Y(R)},\text{ where }I(z)=\int_X\kappa(F(x),z)d\mu_X(x),
\end{equation}
and $v_X,\, V_X,\, v_Y$, and $V_Y$, denote the functions defined in (\ref{bgeom}) for $X$ and $Y$ respectively.

\medskip
\noindent We let now 
$$R_X=2\left(\frac{R-2\sf{c}}{\sf{\Lambda}}\right)+\sf{\Lambda}\left(5\sf{c}_Y+1+3\sf{c}\right),$$
and consider the kernel $\kappa_X$ on $X$ given by (\ref{kernel}) with $r=R_X$. Finally, define the kernel $\kappa_Y$ on $Y$ given by
$$\kappa_Y(z,z')=\frac{1}{I(z)}\int_{X^2}\kappa_X(x,y)\kappa(F(x),z)\kappa(F(y),z')d\mu_X(x)d\mu_X(y).$$
One checks that $\kappa_Y$ is bounded (by (\ref{est1})), that its positivity radius is at least $5\sf{c}_Y+1$, and that $\kappa_Y(z,z')=0$ if $d(z,z')>2R+\sf{\Lambda} R_X+\sf{c}$.

\medskip
\noindent By construction, and by Jensen's inequality, $\|(\omega*\kappa)\circ F\|_{\phi,\kappa_X}\leq \|\omega\|_{\phi,\kappa_Y}$ for any cocycle $\omega$ on $Y$. This shows that $F^*$ maps $Z_{\phi,\kappa_Y}(Y)$ on $Z_{\phi,\kappa_X}(X)$. Moreover, by (\ref{est1}) and Jensen's inequality again, one checks that $\|(u*\kappa)\circ F\|_\phi\lesssim \|u\|_{\phi}$ for any measurable function $u:Y\to \R$. Since $D\left((u*\kappa)\circ F\right)=\left((Du)*\kappa\right)\circ F$, we see that $F^*$ induces a continuous map from $L^\phi_{\kappa_Y}H^1(Y)$ to $L^\phi_{\kappa_X}H^1(X)$. Since quasi-isometries at bounded distance induce the same maps in cohomology, this shows that $F^*$ induces an isomorphism.
\end{proof}

\subsection{Comparision between the De Rham and the simplicial Orlicz cohomologies}\label{riemann}

In this section we identify the De Rahm cohomology with the coarse cohomology for a simply connected complete Riemannian manifold $M$ with Ricci curvature bounded from below and positive injectivity radius. We denote by $dx$ the volume element of $M$ and notice that $M$ satisfies conditions (\ref{bgeom}). Also, $M$ is quasi-isometric to a geometric simplicial graph $X_M$, see \cite{MasahikoKanai85}. 

\medskip
\noindent Let $Z_\phi^\infty(M)$ be the space of closed smooth differential $1$-forms in $M$. Equip it with the Luxembourg norm: for $\omega\in Z^\infty(M)$,
$$\|\omega\|_\phi=\inf\left\{\alpha>0:\int_M\phi\left(\frac{|\omega_x|}{\alpha}\right)dx\leq 1\right\}\in[0,+\infty].$$
Here, for $x\in M$, $|\omega_x|$ is the operator norm of the linear map $\omega_x:T_xM\to\R$. Denote $Z^\infty_\phi(M):=\{\omega\in Z^\infty(M):\|\omega\|_\phi<\infty\}$. Its completion $Z_\phi(M)$ is the space of measurable $\phi$-integrable $1$-forms on $M$ which have zero weak exterior derivative.

\medskip
\noindent The exterior derivative $d:C^\infty_\phi(M)\to Z^\infty_\phi(M)$ is a bounded operator when we consider the norm $\|u\|_\phi+\|du\|_\phi$ in the space of smooth functions. Let the operator $d$ to be defined in the completion of $C_\phi^\infty(M)$; that is, the space of $\phi$-integrable functions on $M$ which have $\phi$-integrable weak derivative. We define the Orlicz-De Rham first cohomology group of $M$ as the quotient normed space $L^\phi H^1(M):=Z_\phi(M)/\mathrm{Im}\,d$. By integration, we can show that 
$$L^\phi H^1(M)\simeq \left\{u\in L^1_{\mathrm{loc}}(M):\|du\|_\phi<\infty\right\}/\left\{u\in L^\phi(M):\|du\|_\phi<\infty\right\}\oplus \R.$$
The next lemma is a minor generalization of \cite[Lemma 1.\,12]{PierrePansu89}.
\begin{lemma}\label{Dvsd}
For any $0<r<\mathrm{inj}(M)$, there exist a kernel $\kappa_r$ on $M$ and a constant $\sf{C}_r$ such that for any $u\in L^1_{\mathrm{loc}}(M)$ with $\|du\|_\phi<\infty$, the following inequality holds
$$\|Du\|_{\phi,\kappa_r}\leq \sf{C}_r \|du\|_\phi.$$
The constant $\sf{C}_r$ depends only on $r$ and $\phi$.
\end{lemma}
\begin{proof}
For a point $x\in M$, let $(v,t)\in T_x^1M\times (0,r]$, $y=\exp(tv)$, be the polar coordinates with origin $x$. Let $\pi:T^1M\to M$ be the canonical projection. We denote by $\varrho(x,y)$ the volume element in that coordinates, so $\varrho(x,y)^{-1}dy=dtdv$. Notice that $\varrho(x,y)=\varrho(y,x)$ a.e., see \cite[Lemma 1.\,12]{PierrePansu89}. Let also $\varphi_s$ be the geodesic flow.

\medskip
\noindent By the assumptions on $u$, there exists a sequence of smooth functions $u_n$ on $M$ which converges to $u$ in $L^\phi_{\mathrm{loc}}(M)$ and a.e., and such that $du_n$ converges to $du$ in $L^\phi(M)$. For each $n\in\N$, $x\in M$ and $v\in T^1M$, we have
\begin{equation}\label{smooth}
\left|u_n\left(\pi\left(\varphi_t(v)\right)\right)-u_n(x)\right|\leq \int_0^t\left|du_n\left(\varphi_s(v)\right)\right|ds.
\end{equation}
Taking limit as $n\to \infty$, $u$ satisfies (\ref{smooth}) a.e.. Note that by convexity, the function $\phi(t)/t$ is increasing. For $\alpha>0$, Jensen's inequality implies
$$\phi\left(\frac{u\left(\pi\left(\varphi_t(v)\right)\right)-u(x)}{\alpha}\right)\leq \frac{1}{r}\int_0^t\phi\left(\frac{r\left|du\left(\varphi_s(v)\right)\right|}{\alpha}\right)ds,$$
Integrating this inequality over $T_x^1M\times (0,r]$, we get
$$\int_{B(x,r)}\phi\left(\frac{u(y)-u(x)}{\alpha}\right)\varrho(x,y)^{-1}dy\leq \int_{T_x^1M}\int_0^r\phi\left(\frac{r|d_{\pi\varphi_t(v)}u|}{\alpha}\right)dtdv,\text{ for a.e. }x\in M.$$
Integrating now over $x\in M$, we obtain
$$\int_{d(x,y)\leq r}\phi\left(\frac{u(y)-u(x)}{\alpha}\right)\varrho(x,y)^{-1}dxdy\leq \mathrm{Vol}(S^{n-1})r\int_M\phi\left(\frac{r|d_xu|}{\alpha}\right)dx.$$
Letting $\kappa_r(x,y)=\min\left\{1,\varrho(x,y)^{-1}\right\}\I_{\{d(x,y)\leq r\}}$, we obtain $\|Du\|_{\phi,\kappa_r}\lesssim\|du\|_\phi$, with a multiplicative constant depending only on $r$ and $\phi$.
\end{proof}

\medskip
\noindent In the statement of the next proposition, we denote $\kappa_y:M\to\R$, with $y\in M$, the function $\kappa_y(x)=\kappa(x,y)$, where $\kappa$ is a kernel on $M$. The next proposition is a generalization to the general Orlicz case of \cite[Prop.\ 1.\,13]{PierrePansu89}.

\begin{proposition}\label{derham}
Let $M$ be a simply connected complete Riemannian manifold with Ricci curvature bounded from below and positive injectivity radius. Let $\kappa$ be a smooth kernel on $M$ satisfying in addition that
$$\sf{C}=\|d\kappa\|_\infty<\infty\text{ and } |d_x\kappa_z|\geq \delta>0\text{ if }d(x,z)\leq \varepsilon_\kappa.$$
Then there exists a kernel $\tilde{\kappa}$ on $M$, depending only on $\kappa$, such that the map $u\mapsto u*\kappa$ induces an isomorphism between $L^\phi_{\tilde{\kappa}} H^1(M)$ and $L^\phi H^1(M)$. In particular, the Orlicz-De Rham cohomology of $M$ is isomorphic to the simplicial cohomology $\ell^\phi H^1(X_M)$.
\end{proposition}
\begin{proof}
Consider the kernel $\tilde{\kappa}$ on $M$ given by
$$\tilde{\kappa}(z,z'):=\frac{1}{I_z}\int_M|d_x\kappa_z|\kappa(x,z')dx,\text{ where } I_z=\int_M|d_x\kappa_z|dx.$$
Let $\sf{a}=\delta v(\varepsilon_\kappa)$, and $\sf{b}=\sf{C} V(R)$. Notice that $\sf{a}\leq I_z\leq \sf{b}$, and that
$$\int_{M^2}|d_x\kappa_z|\kappa(x,z')dzdz'=\int_M|d_x\kappa_z|dz\in\left[\sf{a},\sf{b}\right].$$
Let $u\in L^1_{\mathrm{loc}}(M)$ with $\|Du\|_{\phi,\tilde{\kappa}}<\infty$. Consider the cocycle $\omega=(Du)*\kappa=D(u*\kappa)$. For any $y\in M$, the smooth function $u*\kappa$ satisfies $d(u*\kappa)=d\omega_y$, where $\omega_y(x)=\omega(x,y)$. Then
$$|d_x(u*\kappa)|\leq \int_{M^2}\left|u(z)-u(z')\right||d_x\kappa_z|\kappa(x,z')dzdz', \text{ for all } x\in M.$$
By Jensen's inequality, for any $\alpha>0$, we get
$$\int_M\phi\left(\frac{|d_x(u*\kappa)|}{\alpha}\right)dx\leq \frac{\sf{b}}{\sf{a}}\int_{M^2}\phi\left(\frac{\sf{b}|u(z)-u(z')|}{\alpha}\right)\tilde{\kappa}(z,z')dzdz'.$$
That is, $\|d(u*\kappa)\|_\phi\lesssim \|Du\|_{\phi,\tilde{\kappa}}$, where the comparison constant depends only on $\kappa$, $\phi$, and the geometric data of $M$.

\medskip
\noindent Consider now the bounded operator
\begin{align*}
\left\{u\in L^1_{\mathrm{loc}}(M):\|Du\|_{\phi,\tilde{\kappa}}<\infty\right\}\to \left\{u\in L^1_{\mathrm{loc}}(M):\|du\|_{\phi}<\infty\right\},\ u\mapsto u*\kappa
\end{align*}
If $u\in L^\phi(M)$, then $\|Du\|_{\phi,\tilde{\kappa}}<\infty$, and therefore $\|d(u*\kappa)\|_\phi\lesssim \|Du\|_{\phi,\tilde{\kappa}}<\infty$. This implies that $T:L^\phi_{\tilde{\kappa}}H^1(M)\to L^\phi H^1(M)$, $[u]\mapsto [u*\kappa]$ is a well defined bounded operator.

\medskip
\noindent We show now that this map is an isomorphism. Let $u\in L^1_{\mathrm{loc}}(M)$ be such that $\|du\|_\phi<\infty$. By Lemma \ref{Dvsd}, we have $\|Du\|_{\phi,\tilde{\kappa}}<\infty$. Moreover, since 
\begin{equation}\label{reg}
u*\kappa(x)-u(x)=\int_M(u(y)-u(x))\kappa(x,y)dy,
\end{equation}
by Jensen's inequality we get $\|u*\kappa-u\|_\phi\lesssim \|Du\|_{\phi,\tilde{\kappa}}$. Then, 
$$\|u*\kappa-u\|_\phi+\|d(u*\kappa-u)\|_\phi\lesssim \|Du\|_{\phi,\tilde{\kappa}}+\|du\|_\phi<\infty.$$
That is, $[u*\kappa]=[u],$ and in particular, $T$ is surjective.

\medskip
\noindent Suppose now that $u\in L^1_{\mathrm{loc}}(M)$, $\|Du\|_{\phi,\tilde{\kappa}}<\infty$, and $[u*\kappa]=0$ in $L^\phi H^1(M)$. Since we also have $\|Du\|_{\phi,\kappa}<\infty$, we can apply (\ref{reg}) and conclude that $u*\kappa-u\in L^\phi(M)$. Then, $u\in L^\phi(M)$ and $[u]=0$ in $L^\phi_{\tilde{\kappa}}(M)$. This shows that $T$ is an isomorphism.
\end{proof}

\begin{remark} Note that the proof of the previous proposition also shows that in each cohomology class there is a smooth function: in fact, since $\kappa$ is a smooth kernel the convolution $u*\kappa$ is a smooth function and $[u*\kappa]=[u]$ in $L^\phi H^1(M)$.
\end{remark}

\section{Radial limits and Orlicz-Besov spaces}\label{boundary}

\noindent In this section we prove Theorem \ref{thetracemap}. In the sequel, we assume that $X$ is a uniformly contractible, Gromov hyperbolic, quasi-starlike simplicial complex with bounded geometry. We assume also that $\partial X$ admits a visual metric $\varrho_w$ of visual parameter $e$, where $w\in X$ is a base point, which is Ahlfors regular of dimension $Q>0$. Recall that $H$ denotes the $Q$-dimensional Hausdorff measure of $\varrho_w$.

\begin{remark}
The contractibility assumption is not a restriction in the hyperbolic case. Indeed, a well known theorem of E. Rips \cite[Section 3.23]{Bridson-Haefliger99}, assures that there is always a uniformly contractible Rips complex $\tilde{X}$ of $X$. The inclusion map of $X$ into $\tilde{X}$ is a quasi-isometry, and $\tilde{X}_0=X_0$. In particular, by Theorem \ref{invariance}, the space $\Lp^\phi H^1(\tilde{X})$ does not depend on the particular choice of $\tilde{X}$ up to isomorphism of topological vector spaces.
\end{remark}

\noindent Since $X$ is contractible, the differential operator $\delta_0$ induces an isomorphism
\begin{align*}\label{isom}\Lp^\phi H^1(X)\simeq\left\{f:X_0\to\R:\delta_0f\in\Lp^\phi(X_1)\right\}/(\Lp^\phi(X_0)+\R),\end{align*}
where $\R$ denotes the constant functions on $X_0$. Here, the space on the right is equipped with the topology induced by $\|\delta_0f\|_\phi$. In this section we will write $d$ instead of $\delta_0$.

\subsection{Spherical coordinates and radial shift}

Let $\mathcal{R}_w$ be the compact space of geodesic rays starting at $w$, and let $\pi_w:\mathcal{R}_w\to \partial X$ be the projection $\pi_w(\theta)=\lim_{r\to+\infty}\theta(r)$. Consider $\theta:\partial X\to \mathcal{R}_w$, $\xi\mapsto \theta_\xi$, a measurable section of $\pi_w$. For $r\geq 1$, $\theta_\xi[r]$ denotes the edge of $\theta_\xi$ at distance $r-1$ from $w$.

\medskip
\noindent We use $|x|$ for the distance $|x-w|$. For $r\in \N$, let $\Sigma_r=\{|x|=r\}$, and $A_r$ be the set of edges of $X$ with one extremity in the sphere $\Sigma_{r-1}$ and the other one in $\Sigma_r$. If $\sigma\in A_r$, we denote by $\sigma_+$ (resp. $\sigma_-$) the extremity of $\sigma$ in $\Sigma_r$ (resp. $\Sigma_{r-1}$).

\medskip
\noindent Define the ``volume element'' by $v:=e^Q$, and let $\mu_v$ be the discrete measure on $\N$ assigning the mass $v^r$ at $r\in\N$. In the product space $\partial X\times \N$, we consider the product measure $H\otimes \mu_v$. When a $\phi$-norm is computed in this product space, it will be understood that it is with respect to this measure. By the spherical coordinates on $X$, we mean the map $\partial X\times \N\to X_1$ given by $(\xi,r)\mapsto \theta_\xi[r]$.

\medskip
\noindent Consider the radial shift map $S:\partial X\times\N\to\partial X\times\N$ given by $S(\xi,r)=(\xi,r+1)$. If $G:\partial X\times\N\to\R$ is any function, we write $S^*G$ for the function $G\circ S$.

\begin{lemma}\label{contraction}
Let $\phi$ be a doubling Young function with doubling constant $\sf{K}$. The radial shift defines a bounded operator $S^*:L^\phi(\partial X\times\N)\to L^\phi(\partial X\times\N)$ with $\|S^*\|\leq v^{-1/\sf{K}}$.
\end{lemma}
\begin{proof}
By Lemma \ref{lowerdoubling}, we have $x^\sf{K}\phi(y)\leq \phi(xy)$ for any $x\in [0,1]$. Let $G\in L^\phi(\partial X\times\N)$ and take $\alpha>0$. Then
\begin{align*}
\int_{\partial X}\sum_{r=1}^\infty\phi\left(\frac{S^*G(\xi,r)}{\alpha}\right)v^rdH(\xi)&
=\int_{\partial X}\sum_{r=2}^\infty \frac{1}{v}\phi\left(\frac{G(\xi,r)}{\alpha}\right)v^rdH(\xi)\\
&\leq \int_{\partial X}\sum_{r=1}^\infty\phi\left(\frac{G(\xi,r)}{v^{1/\sf{K}}\alpha}\right)v^rdH(\xi).
\end{align*}
Then, for $\alpha=v^{-1/\sf{K}}\|G\|_\phi$, the last integral is equal to one. This proves the lemma.
\end{proof}

\noindent An immediate consequence of Lemma \ref{contraction} is that the operator 
\begin{equation}\label{T}
T:=\sum_{k=0}^\infty (S^*)^k
\end{equation}
is bounded. Its norm is bounded by a uniform constant depending only on $\sf{K}$ and $v$.

\subsection{The radial limit} Let $f:X_0\to \R$ be such that $df\in\Lp^\phi(X_1)$. For $\xi\in \partial X$, we set
\begin{align}\label{limit}
f_\infty(\xi)=\lim_{r\to \infty}f(\theta_\xi(r)),
\end{align}
if the limit exists. To make the notation simpler, we write $\Delta_{\xi,r}=df(\theta_\xi[r])$, $r\in\N$. Define $F:\partial X\times\N\to \R$ and $DF:\partial X\times\N\to\R$ to be the functions given by
$$F(\xi,r)=\sum_{k\geq r}\Delta_{\xi,k}\text{ and } DF(\xi,r)=\Delta_{\xi,r}.$$
Our goal is to show that $F$ is well defined. We first need some notation.

\medskip
\noindent For $x\in X$ and $R>0$, define the shadow 
$$\mho_w(x,R):=\left\{\xi\in\partial X: \exists\,\theta\in\mathcal{R}_w,\text{ asymptotic to }\xi,\text{ with }\theta\cap B_X(x,R)\neq\emptyset\right\}.$$
For $R$ big enough, there is a constant $\sf{C}_R\geq 1$, such that for any $x\in X$, there exists $\xi\in\partial X$ with
$$B\left(\xi,\sf{C}_R^{-1}e^{-|x|}\right)\subset \mho_w(x,R)\subset B\left(\xi,\sf{C}_Re^{-|x|}\right).$$
See for example \cite{Coornaert93}. We fix such a big enough $R$ and we write $\mho(x)$ instead of $\mho_w(x,R)$. Notice that $H\left(\mho(x)\right)\asymp v^{-|x|}$.

\begin{lemma}\label{existence3}
Let $f:X_0\to\R$ be any function. Then, $\|DF\|_\phi\lesssim \|df\|_\phi$. In particular, if $df\in \Lp^\phi(X_1)$, then $DF\in L^\phi(\partial X\times\N)$.
\end{lemma}
\begin{proof}
We first prove that there exists a constant $\sf{C}\geq 1$, such that for any function $f:X_0\to \R$ and any $\alpha>0$, we have
\begin{equation}\label{existence2}
\int_{\partial X}\sum_{r=1}^\infty\phi\left(\frac{\Delta_{\xi,r}}{\alpha}\right)v^{r}dH(\xi)\leq \sf{C}\sum_{\sigma\in X_1}\phi\left(\frac{df(\sigma)}{\alpha}\right).
\end{equation}
For a fixed edge $\sigma\in A_r$, the set of points $\xi\in \partial X$ such that $\theta_\xi[r]=\sigma$, is contained in the shadow $\mho(\sigma_+)$. Then, there is a uniform constant $\sf{C}$ such that $H\left(\xi:\theta_\xi[r]=\sigma\right)\leq \sf{C}v^{-r}$. Applying Fubini's theorem, this implies that
\begin{align*}
\int_{\partial X}\sum_{r=1}^\infty\phi\left(\frac{\Delta_{\xi,r}}{\alpha}\right)v^{r}dH(\xi)&=\sum_{r=1}^{\infty}\sum_{\sigma\in A_r}\int_{\left\{\xi:\theta_\xi[r]=\sigma\right\}}\phi\left(\frac{df(\sigma)}{\alpha}\right)v^{r}dH(\xi)\\
&\leq  \sf{C}\sum_{r=1}^{\infty}\sum_{\sigma\in A_r}\phi\left(\frac{df(\sigma)}{\alpha}\right)\leq \sf{C}\sum_{\sigma\in X_1}\phi\left(\frac{df(\sigma)}{\alpha}\right).
\end{align*}
When $\|df\|_\phi<\infty$, it suffices to take $\alpha>0$ so that the sum on the right hand side of (\ref{existence2}) is less than or equal to $\sf{C}^{-1}$.
\end{proof}

\medskip
\noindent Notice that 
$$F(\xi,r)=\sum_{k=r}^{\infty}\Delta_{\xi,k}=\sum_{k=0}^\infty\Delta_{\xi,r+k}=\left(\sum_{k=0}^\infty (S^*)^kDF\right)(\xi,r),$$
that is, $F=T(DF)$. As an immediate consequence, an inequality ``\`a la Strichartz'' holds (compare with \cite[Lemma 3.3]{Bourdon-Pajot04} and \cite{Strichartz83}): there exists a constant $\sf{C}:=\|T\|$, depending only on $\phi$ and $v$, such that for any $f:X_0\to\R$ with $df\in\Lp^\phi(X_1)$, we have
\begin{equation}\label{strichartz}
\|F\|_\phi\leq C\|DF\|_\phi.
\end{equation}
This will be a key point later in the proof of Theorem \ref{thetracemap}.

\medskip
\noindent We summarize the consequences of inequality (\ref{strichartz}) in the next corollary.  For each $r\geq 1$, define the function $f_r:\partial X\to \R$ by setting $f_r(\xi)=f(\theta_\xi(r-1))$. Note that the image of $f_r$ is a finite set.

\begin{corollary}[The radial limit]\label{extension}
Let $f:X_0\to\R$ be a function such that $df\in \Lp^\phi(X_1)$. Then the radial limit $f_\infty(\xi)$ exists $H$-almost everywhere. Moreover: 
\begin{enumerate}
\item the function $f_\infty\in L^\phi(\partial X,H)$,
\item $f_r\to f_\infty$ in $L^\phi(\partial X)$ when $r\to\infty$, and $\|f_\infty-f(w)\|_\phi\lesssim \|df\|_\phi$.
\end{enumerate}
\end{corollary}
\begin{proof}
First, note that $F(\cdot,r)=f_\infty-f_r$. Then, for any $\alpha>0$, we have
$$\sum_{r=1}^\infty\left[\int_{\partial X}\phi\left(\frac{f_\infty(\xi)-f_r(\xi)}{\alpha}\right)dH(\xi)\right]v^r=\int_{\partial X}\sum_{r=1}^\infty\phi\left(\frac{F(\xi,r)}{\alpha}\right)v^rdH(\xi).$$
By letting $\alpha=\|F\|_\phi$ and noticing that $f_1\equiv f(w)$, we obtain
$$\|f_\infty-f(w)\|_\phi\leq\|F\|_\phi\lesssim \|df\|_\phi.$$
Furthermore, taking $\alpha=1$ in the above equality, we see that
$$\int_{\partial X}\phi\left(f_\infty(\xi)-f_r(\xi)\right)dH(\xi)\to 0,\text{ when }r\to\infty.$$
Since $\phi$ is doubling, $\|f_\infty-f_r\|_\phi\to 0$ when $r\to\infty$, and the limit is also in $L^\phi(\partial X)$.
\end{proof}

\begin{remark} The existence of the radial limit can be proved under the slightly weaker assumption of $\phi$ being of polynomial decay. The cost is that the limit is in $L^1(\partial X,H)$. On the other hand, the radial limit does not necessarily exists when $\phi$ decays very fast (like the function $\phi(t)=t^p\exp(-1/t)$, $p\geq 1$). The constraints for a function to be $\phi$-integrable in that case are very weak near infinity. Nevertheless, this kind of functions could be useful in situations were the boundary is not well defined, for instance, in absence of hyperbolicity, or when the space has infinite growth ($v=\infty$). 
\end{remark}

\subsection{Proof of Theorem \ref{thetracemap}}

Theorem \ref{extension} ensures the existence of a continuous map
$$\mathcal{T}:\left\{f:X_0\to\R:df\in\Lp^\phi(X_1)\right\}/\R\to L^\phi(\partial X,H)/\R,\ f\mapsto f_\infty.$$
In this section we determine the image and the kernel of $\mathcal{T}$. 

\begin{proposition}[The kernel of $\mathcal{T}$]\label{nucleo}
The kernel of $\mathcal{T}$ is precisely $(\Lp^\phi(X_0)\oplus\R)/\R$. In particular, the $\Lp^\phi$-cohomology is reduced in degree one.
\end{proposition}

\noindent The proof follows from the next lemma which compares the $\phi$-norms of $f$ and $F$.

\begin{lemma}\label{bigF}
Let $\phi$ be any Young function, and let $f:X_0\to \R$ be such that $df\in\Lp^\phi(X_1)$ and $f_\infty=0$ $H$-almost everywhere. Then, $f\in \Lp^\phi(X_0)$ if and only if $F\in L^\phi(\partial X\times \N)$.
\end{lemma}

\begin{proof}[Proof of Proposition \ref{nucleo}]
It is clear that if $f\in\Lp^\phi(X_0)$, then $f_\infty\equiv0$. Suppose that $f_\infty=0$ $H$-a.e.\ and that $df\in\Lp^\phi(X_1)$. Then, $F\in L^\phi(\partial X\times\N)$ by Lemma \ref{existence3} and inequality (\ref{strichartz}). Therefore, $f\in \Lp^\phi(X_0)$ by Lemma \ref{bigF}.
\end{proof}

\noindent We will need for the proof of Lemma \ref{bigF}, the fact that the shadows $\{\mho(x):x\in \Sigma_r\}$ form a family of coverings of $\partial X$ with the following nice properties:
\begin{enumerate}
\item There is a small constant $\sf{c}>0$ such that for each $r\in\N$, the shadow $\mho(x)$, with $x\in \Sigma_r$, contains a ball $B(x)$ of radius $\sf{c}v^{-r}$, and such that for each $r\in\N$, the balls $\{B(x):x\in\Sigma_r\}$ are pairwise disjoint. 
\item There is a constant $\sf{D}$ such that if $\xi\in \mho(x)$, with $x\in\Sigma_r$, then $|\theta_\xi(r)-x|\leq \sf{D}$.
\item For any $\xi\in \partial X$ and $r\in\N$, the number 
\begin{equation}\label{overlap}
m_{\xi,r}=\#\left\{x\in \Sigma_r: \xi\in \mho(x)\right\}\leq \sf{M},
\end{equation}
for some uniform constant $\sf{M}$.
\end{enumerate}

\begin{proof}[Proof of Lemma \ref{bigF}]
Note that since $f_\infty=0$ almost everywhere, then $F(\xi,r)=f_r(\xi)$ almost everywhere. Denote by $V(\xi,r)$ the ball of radius $\sf{D}$ in $X$ centered at $\theta_\xi(r)$. By the convexity of $\phi$, for any $\alpha>0$, we have
\begin{equation}\label{discretevscont}
\phi\left(\frac{f(x)-f_{r}(\xi)}{\sf{D}\alpha}\right)\leq \sum_{\sigma\subset V(\xi,r)}\phi\left(\frac{df(\sigma)}{\alpha}\right)\text{ for all } x\in \Sigma_r,\ \xi\in \mho(x),\ r\in\N.
\end{equation}
Define $\tilde{F}:\partial X\times \N\to \R$ as the average $\tilde{F}(\xi,r)=\frac{1}{m_{\xi,r}}\sum_{x\in \Sigma_r}f(x)\I_{\mho(x)}(\xi)$. The convexity of $\phi$ again and (\ref{discretevscont}) imply that
$$\phi\left(\frac{\tilde{F}(\xi,r)-F(\xi,r)}{\sf{D}\alpha}\right)\leq \sum_{\sigma\subset V(\xi,r)}\phi\left(\frac{df(\sigma)}{\alpha}\right), \text{ for any }\xi\in \partial X,\ r\in \N.$$
Note that if $\xi\in \mho(y)$ for $y\in \Sigma_r$, then $V(\xi,r)\subset B(y,2\sf{D})$. Integrating the last inequality in $\mho(y)$ first, and then summing over $y\in X_0$, we obtain 
\begin{align*}
\sum_{r=1}^\infty\int_{\partial X}\phi\left(\frac{\tilde{F}(\xi,r)-F(\xi,r)}{\sf{D}\alpha}\right)v^{r}dH(\xi)&\leq \sum_{r=1}^\infty\sum_{y\in \Sigma_r}\int_{\mho(y)}\phi\left(\frac{\tilde{F}(\xi,r)-F(\xi,r)}{\sf{D}\alpha}\right)v^{r}dH(\xi)\\
&\lesssim\sum_{r=1}^\infty\sum_{y\in \Sigma_r}\sum_{\sigma\in B(y,2\sf{D})}\phi\left(\frac{df(\sigma)}{\alpha}\right)\\
&\lesssim \sum_{\sigma\in X_1}\phi\left(\frac{df(\sigma)}{\alpha}\right),\end{align*}
where the multiplicative constants depend only on the Ahlfors regularity  of $H$, the degree of $X$, and $\phi$. Therefore, $\|\tilde{F}-F\|_{\phi}\lesssim\|df\|_{\phi}<+\infty$. That is, $F\in L^\phi(\partial X\times\N)$ if, and only if, $\tilde{F}$ is. 

\medskip
\noindent We end the proof by comparing the norms of $f$ and $\tilde{F}$. Recall from the definition of the constant $\sf{M}$ in (\ref{overlap}), that we have $m_{\xi,r}\leq \sf{M}$. Then, if $x\in \Sigma_r$, we have
\begin{align*}
\phi\left(\frac{f(x)}{\sf{M}\alpha}\right)&\lesssim\int_{B(x)}\phi\left(\frac{\tilde{F}(\xi,r)}{\alpha}\right)v^rdH(\xi)\\
&\leq\int_{\mho(x)}\phi\left(\frac{\tilde{F}(\xi,r)}{\alpha}\right)v^rdH(\xi)\lesssim\sum_{y\in B(x,2\sf{D})}\phi\left(\frac{f(y)}{\alpha}\right).\end{align*}
Summing over $x\in \Sigma_r$ first, and then over $r\in \N$, the same arguments as before show that $f\in \Lp^\phi(X_0)$ if, and only if, $\tilde{F}\in L^\phi(\partial X\times\N)$. This finishes the proof.
\end{proof}

\noindent We now describe the image of $\mathcal{T}$. In the proof of the next proposition, we use the following estimate for the density of the measure $\lambda$ in the space of pairs $\partial^2 X$ (see the Introduction):
\begin{equation}\label{invdistance}
\frac{1}{\rho(\xi,\zeta)^{2Q}}\asymp \sum_{\sigma\in G_1}\frac{\I_{\mho(\sigma_-)}(\xi)\I_{\mho(\sigma_+)}(\zeta)}{H\left(\mho(\sigma_-)\right)H\left(\mho(\sigma_+)\right)}\end{equation}
See \cite[Lemma 3.5]{Bourdon-Pajot04} for the proof.

\begin{proposition}[The image of $\mathcal{T}$]\label{image}
The image of $\mathcal{T}$ is precisely the Orlicz-Besov space $B^\phi(\partial X,\varrho_w)/\R$, and $\mathcal{T}$ is continuous in the Orlicz-Besov norm.
\end{proposition}
\begin{proof}
Let $u$ be a measurable function on $\partial X$ with $\langle u\rangle_\phi<\infty$. Consider the function $f:X_0\to\R$ given by the averages of $u$ over the shadows:
$$f(x):=\dashint_{\mho(x)}u(\xi)dH(\xi).$$
By Lebesgue's differentiation theorem, the radial limit $f_\infty$ exists $H$-almost everywhere and is equal to $u$. We must show that $df\in \Lp^\phi(X_1)$. Let $\alpha>0$, by Jensen's inequality
\begin{align*}
\phi\left(\frac{df(\sigma)}{\alpha}\right)&=\phi\left(\dashint_{\mho(\sigma_-)\times\mho(\sigma_+)}\frac{u(\xi)-u(\zeta)}{\alpha}dH(\xi) dH(\zeta)\right)\\
&\leq \dashint_{\mho(\sigma_-)\times\mho(\sigma_+)}\phi\left(\frac{u(\xi)-u(\zeta)}{\alpha}\right)dH(\xi) dH(\zeta).
\end{align*}
Summing over $\sigma\in X_1$, and using the inequalities (\ref{invdistance}), we obtain
$$\sum_{\sigma\in X_1}\phi\left(\frac{df(\sigma)}{\alpha}\right)\lesssim \int_{\partial^2 X}\phi\left(\frac{u(\xi)-u(\zeta)}{\alpha}\right)d\lambda(\xi,\zeta).$$
This shows that $df\in \Lp^\phi(X_1)$ and $\|df\|_\phi\lesssim \langle u\rangle_\phi$.

\medskip
\noindent Conversely, let $f:X_0\to\R$ be such that $df\in\Lp^\phi(X_1)$. Let $\alpha>0$ and $\sigma\in X_1$. Then, by the convexity of $\phi$, we have
\begin{align*}
\dashint_{\mho(\sigma_-)\times\mho(\sigma_+)}\phi\left(\frac{f_\infty(\xi)-f_\infty(\zeta)}{3\alpha}\right)dH(\xi)dH(\zeta)\leq&\\
\dashint_{\mho(\sigma_-)}\phi\left(\frac{F(\xi,|\sigma_-|)}{\alpha}\right)dH(\xi)\ +&\ \dashint_{\mho(\sigma_+)}\phi\left(\frac{F(\zeta,|\sigma_+|)}{\alpha}\right)dH(\zeta)+\phi\left(\frac{df(\sigma)}{\alpha}\right).
\end{align*}
In order to bound from above the right hand side of the last inequality, notice that if $\sf{M}$ denotes the constant defined in (\ref{overlap}), then for any $r\in\N$, we have
$$\sum_{x\in \Sigma_r}\I_{\mho(x)}\leq \sf{M}\left(1+\sum_{x\in \Sigma_r}\I_{B(x)}\right).$$
Thus,
\begin{align*}
\sum_{r=1}^\infty\sum_{x\in \Sigma_r}\int_{\mho(x)}\phi\left(\frac{F(\xi,r)}{\alpha}\right)v^rdH(\xi)&\leq \sf{M}\left(1+\sum_{r=1}^\infty\sum_{x\in \Sigma_r}\int_{B(x)}\phi\left(\frac{F(\xi,r)}{\alpha}\right)v^rdH(\xi)\right)\\
&\leq \sf{M}\left(1+\sum_{r=1}^\infty\int_{\partial X}\phi\left(\frac{F(\xi,r)}{\alpha}\right)v^rdH(\xi)\right).
\end{align*}
By Lemma \ref{existence3} and inequality (\ref{strichartz}), the know that $\|F\|_\phi\lesssim \|df\|_\phi$. Then $\langle f_\infty\rangle_\phi\lesssim \|df\|_\phi$. This finishes the proof.
\end{proof}

\section{Localization}\label{localization}

\noindent In this section we prove Theorem \ref{thelocaltracemap}. We will always assume that $X$ is a uniformly contractible, quasi-starlike, Gromov hyperbolic simplicial complex, with bounded geometry, and that $\phi$ is a doubling Young function.

\subsection{Horospherical coordinates and hororadial shift}\label{horoshift}

We first recall the definition and properties of the parabolic visual metrics defined on the punctured boundary $\partial_\xi X:=\partial X\setminus \{\xi\}$. For a geodesic ray $r:(-\infty,0]\to X$, denote the horofunction (or Busemann function) at $r$ by $b_r:X\to \R$, and recall thar the parabolic Gromov product from $r$ is given by
$$(x,y)_r:=\frac{1}{2}\left(b_r(x)+b_r(y)-d(x,y)\right).$$
Denote by $\xi=r(-\infty)\in \partial X$. This product can be extended to $\partial_\xi X$ in the same way as the usual Gromov product, and the function $\varrho_r(\zeta,\eta)=e^{-(\zeta,\eta)_r}$ is a quasi-metric on $\partial_\xi X$. Without loss of generality, we assume that $\varrho_r$ is a distance. We call it a parabolic visual metric. It is comparable on compact subsets of $\partial_\xi X$ to the usual global visual metric.

\medskip
\noindent This metric is well adapted to the study $\mathrm{Isom}(X,\xi)$, the group of isometries of $X$ which fix the point $\xi$. More precisely, suppose that $F\in \mathrm{Isom}(X,\xi)$, and define the quantity $\sf{\Lambda}_r(F)=e^{-b_r(F(w))}$. Then,
\begin{equation}\label{quasisimilarity}
\varrho_r\left(F(\zeta),F(\eta)\right)\asymp \sf{\Lambda}_r(F)\varrho_r(\zeta,\eta),\ \forall\ \zeta,\eta\in \partial_\xi X.
\end{equation}
That is, $F$ is a quasi-similarity in the metric $\varrho_r$.

\medskip
\noindent We define the shadow from $\xi$ of a ball in $X$ as
$$\mho_\xi(x,R):=\left\{\zeta\in\partial_\xi X:\exists\ \theta=(\xi,\zeta),\ \theta\cap B(x,R)\neq\emptyset\right\},$$
where $(\xi,\zeta)$ denotes a geodesic from $\xi$ to $\zeta$. If the radius $R$ is fixed large enough, there is a uniform constant $\sf{C}\geq 1$ such that, for any $x\in X$, there is a point $\zeta\in \partial_\xi X$ with
\begin{equation}\label{shadow2}
B\left(\zeta,\sf{C}^{-1}e^{-b_r(x)}\right)\subset \mho_\xi(x,R)\subset B\left(\zeta,\sf{C}e^{-b_r(x)}\right).
\end{equation}
From now on, we fix such an $R$ and we write $\mho_\xi(x)$ instead of $\mho_\xi(x,R)$.

\medskip
\noindent We remark that the $(\partial X,\varrho)$ is $Q$-regular if, and only if, $\left(\partial_\xi X,\varrho_r\right)$ is $Q$-regular for some (and therefore for any) $\xi\in \partial X$, see \cite[Section 6]{KevinWildrick08}. We assume $Q$-regularity and denote by $H_r$ the $Q$-dimensional Hausdorff measure of $\varrho_r$. It follows from (\ref{quasisimilarity}), that if $F\in\mathrm{Isom}(X,\xi)$, then $F^*H_r$ is absolutely continuous with respect to $H_r$, and
$$\frac{dF^*H_r}{dH_r}(\zeta)\asymp \sf{\Lambda_r(F)}^Q.$$
This implies that the measure
\begin{equation*}\label{lambdameasure}
d\lambda_r(\zeta,\eta)=\frac{dH_r\otimes dH_r(\zeta,\eta)}{\varrho_r(\zeta,\eta)^{2Q}},
\end{equation*}
is almost preserved by $\mathrm{Isom}(X,\xi)$. In particular, the Orlicz-Besov norm of a measurable function $u:\partial_\xi X\to \R$,
\begin{equation}\label{preservationnorm}
\langle u\rangle_{\phi,\xi}:=\inf\left\{\alpha>0:\int_{\partial^2_\xi X}\phi\left(\frac{u(\zeta)-u(\eta)}{\alpha}\right)d\lambda_r(\zeta,\eta)\leq 1\right\},
\end{equation}
is almost preserved by such isometries. This fact will be important later in Section \ref{hnc}.

\medskip
\noindent By a horosphere centred at $\xi$ we mean a level set $H_c=\{b_r(x)=c\}$ of the horofunction $b_r$. Notice that $H_0$ is the horosphere passing through $r(0)$. Since 
$$b_r(x)\underset{x\to\xi}{\to} -\infty, \text{ and } b_r(x)\underset{x\to\zeta}{\to}+\infty,\ \forall \zeta\in\partial_\xi X,$$
$H_c$ intersects any geodesic from $\xi$ to $\zeta\neq \xi$.

\medskip
\noindent Denote by $\mathcal{G}_\xi$ the set of geodesics $\theta:\R\to X$ with $\theta(-\infty)=\xi$, and $\pi_\xi:\mathcal{G}_\xi\to \partial_\xi X$ the projection $\pi_\xi(\theta)=\theta(+\infty)$. As in Section \ref{boundary}, we denote by $\theta:\partial_\xi X\to\mathcal{G}_\xi$ a measurable section of $\pi_\xi$.

\medskip
\noindent Let $\theta$ be a geodesic in $\mathcal{G}_\xi$. Since $b_r$ is $1$-Lipschitz, there is at least one vertex of $X$ in the intersection $\theta\cap \{c_2\leq b_r\leq c_1\}$ if $c_1-c_2\geq 1$. We parametrize each geodesic $\theta_\zeta$ in such a way that $\theta_\zeta(0)$ is a vertex of $X$ in the set $\{-1/2\leq b_r\leq 1/2\}$. Notice that there is a uniform constant $\sf{C}\geq 1$ such that for any $\zeta\in\partial_\xi X$,
\begin{equation}\label{busgeod}
\left|b_r\left(\theta_\zeta(s)\right)-s\right|\leq \sf{C}+\frac{1}{2}.
\end{equation}
For $\zeta\in\partial_\xi X$ and $h\in \Z$, we denote by $\theta_\zeta[h]:=\sigma$ the edge of $\theta_\zeta$ with extremities 
$$\sigma_-=\theta_\zeta(h-1) \text{ and }\sigma_+=\theta_\zeta(h).$$
As in Section \ref{boundary}, we will also write $v=e^Q$, and we define $\mu_v$ to be the discrete measure in $\Z$ assigning the mass $v^h$ at $h\in\Z$. In the product space $\partial_\xi X\times \Z$, we will always consider the product measure $H_r\otimes \mu_v$. By the horospherical coordinates on $X$ we mean the map $\partial_\xi X\times\Z\to X$ given by $(\zeta,h)\mapsto \theta_\xi[h]$.

\medskip
\noindent The hororadial shift $S:\partial_\xi X\times\Z\to \partial_\xi X\times\Z$, $(\zeta,h)\mapsto (\zeta,h+1)$, induces a contraction $S^*$ on $L^\phi(\partial_\xi X\times\Z)$. The same proof as in Lemma \ref{contraction} holds here. Therefore, the operator $T$ defined as in (\ref{T}), is bounded.

\subsection{Local $\Lp^\phi$-cohomology} 

We first prove that $\ell^\phi_{\mathrm{loc}} H^1(X_k,\xi)$ is a Fr\'echet space, we refer to Definition \ref{deflocal} for the definition. To this end, we start by showing that there are enough quasi-isometric embeddings on $X$ with uniform constants.

\begin{lemma}\label{enoughembeddings}
There exist a constant $\sf{C}\geq 1$, which depends only on the geometric data of $X$, and a sequence $\{Y_n\}_{n\in\N}$ of hyperbolic geometric complexes, such that 
\begin{enumerate}
\item each $Y_n$ is a subcomplex of $X$ with $\partial Y_n\subset \partial_\xi X$, 
\item the inclusion map of $Y_n$ into $X$ is a $(1,\sf{C})$-quasi-isometric embedding.
\end{enumerate}
Moreover, the image of any quasi-isometric embedding $\iota\in\mathcal{Y}(X,\xi)$ is contained in $Y_n$ for $n$ large enough.
\end{lemma}
\begin{proof}
Denote by $\sf{\delta}_{\mathrm{hyp}}$ and $\sf{C}_{\mathrm{st}}$ the hyperbolicity and quasi-starlike constants of $X$. Fix a geodesic ray $r:(-\infty,0]\to X$ with $w=r(0)$ and $\xi=r(-\infty)$. For a compact set $K\subset \partial_\xi X$, we define the cone with base $K$ and height $l\in\Z\cup\{-\infty\}$ as
\begin{equation}\label{cone}
\mathrm{con}(K,l):=\left\{x\in X:b_r(x)\geq l,\text{ and }\exists\text{ a geodesic }(\xi,\zeta)\ni x,\text{ with }\zeta\in K\right\}.
\end{equation}
By hyperbolicity of $X$, the infinite cone $\mathrm{con}(K,-\infty)$ is $\sf{C}_0$-quasiconvex, where $\sf{C}_0$ is a constant depending only on $\sf{\delta}_{\mathrm{hyp}}$. Since all the geodesics defining the cone are asymptotic to $\xi$, if $l:=l_K$ is negative enough, the cone $\mathrm{con}(K,l)$ is $\sf{C}_1$-quasiconvex for a uniform constant $\sf{C}_1$. Consider the subcomplex $Y_{K}$ of $X$ formed by all the simplices of $X$ which intersect the $\sf{C}_2$-neighborhood of $\mathrm{con}(K,l)$, where $\sf{C}_2=\max\{\sf{C}_1,\sf{C}_{\mathrm{st}}\}$.

\medskip
\noindent We equip $Y_{K}$ with the length metric obtained by restricting the length structure of $X$. From the quasiconvexity of $\mathrm{con}(K,l)$, one checks that the inclusion $\iota :Y_{K}\to X$ is a $(1,\sf{C})$-quasi-isometric embedding, where $\sf{C}$ is a constant depending only on $\sf{C}_2$. In particular, $Y_{K}$ is Gromov hyperbolic, with hyperbolicity constant equal to $3\sf{C}+\delta_{\mathrm{hyp}}$. By construction, $\iota_\infty(\partial Y_K)=K$.

\medskip
\noindent Consider the parabolic visual metric $\varrho_r$, and take a point $\zeta\in\partial_\xi X$. For each $n\in\N$, consider the subcomplex $Y_{n}:=Y_{K_n}$ associated to the closed ball $K_n=B_{\rho_r}[\zeta,n]$ in the above construction. They form an increasing sequence of subcomplexes whose union is the whole $X$. One checks that they verify the desired properties \emph{(1)} and \emph{(2)}.
\end{proof}

\noindent It is clear that $\Lp^\phi_{\mathrm{loc}}(X_k,\xi)$ is a vector space. Moreover, by Lemma \ref{enoughembeddings}, it is a Fr\'echet space. In fact, $\left\{\|\cdot\|_{\phi,Y_{n}}:n\in\N\right\}$ is a defining family of seminorms for $\Lp^\phi_{\mathrm{loc}}(X_k,\xi)$, and a translation invariant metric can be defined as usual by
$$D_{\phi,\xi}(\tau_1,\tau_2)=\sum_{n\in\N}\frac{1}{2^{n}}\max\left\{1,\|\tau_1-\tau_2\|_{\phi,Y_{n,l}}\right\}.$$
Let $\iota:Y\to X$ be a quasi-isometric embedding in $\mathcal{Y}(X,\xi)$, and denote by $\delta_k$ and $\delta_{k,Y}$ the coboundary operators of $X$ and $Y$ respectively. Since $\iota^*\circ\delta_k=\delta_{k,Y}\circ \iota^*$, and since the norm of $\delta_{k,Y}:\Lp^\phi(Y_k)\to \Lp^\phi(Y_{k+1})$ is bounded from above by a constant depending only on the geometric data of $Y$, we see that
$$\delta_k:\Lp^\phi_{\mathrm{loc}}(X_k,\xi)\to \Lp^\phi_{\mathrm{loc}}(X_{k+1},\xi)$$
is a Lipschitz continuous operator in the Fr\'echet distance $D_{\phi,\xi}$.

\medskip
\noindent We now focus on the proof of Theorem \ref{thelocaltracemap}. First, we observe that
$$\Lp^\phi_{\mathrm{loc}}H^1(X,\xi)\simeq\left\{f:X_0\to\R:\ df\in \Lp^\phi_{\mathrm{loc}}(X_1,\xi)\right\}\big{/}\left(\Lp^\phi_{\mathrm{loc}}(X_0,\xi)\oplus\R\right).$$
 The approach follows exactly the same arguments as in Section \ref{boundary}, but in this case, we work with the horospherical coordinates on $X$.

\medskip
\noindent If $f:X_0\to\R$ is any function and $\zeta\in\partial_\xi X$, we set $\Delta_{\zeta,h}=df(\theta_\zeta[h]),\ h\in\Z$. Furthermore, let $F:\partial_\xi X\times\Z\to \R$ and $DF:\partial_\xi X\times\Z\to\R$ be the functions defined by
$$F(\zeta,h)=\sum_{k\geq h}\Delta_{\zeta,k}\text{ and } DF(\zeta,h)=\Delta_{\zeta,h}.$$
We denote by $l_n\in \Z$ the height of the cone defining $Y_n$ in Lemma \ref{enoughembeddings}. We write $\Z_n=\{k\in\Z:k\geq l_n\}$. Then, as in the proof of (\ref{strichartz}), we have
$$\|F\|_{\phi,K_n\times\Z_n}\lesssim \|DF\|_{\phi,K_n\times\Z_n},\text{ for all }n\in\N.$$
Note that for a fixed edge $\sigma\in X_1$ and $h\in \Z$, the set of points $\zeta\in \partial_\xi X$ such that $\theta_\zeta[h]=\sigma$ is contained in the shadow $\mho_\xi(\sigma_+)$. In particular, from (\ref{busgeod}) and (\ref{shadow2}), we have
$$H_r\left(\xi:\theta_\xi[h]=\sigma\right)\lesssim v^{-h}.$$
It follows, exactly as for (\ref{existence3}), that
$$\|DF\|_{\phi,K_n\times\Z_n}\lesssim \|df\|_{\phi,Y_{n}},\text{ for all }n\in\N.$$
Therefore, if $f:X_0\to\R$ is such that $df\in \Lp^\phi_{\mathrm{loc}}(X_1,\xi)$, then the hororadial limit $f_\infty$ exists almost everywhere, and $f_\infty\in L^\phi_{\mathrm{loc}}\left(\partial_\xi X,H_r\right)$. Moreover, if we denote by $f_h:\partial_\xi X\to \R$ the function $F(\cdot,h)$, then $f_h\to f_\infty$ in $L^\phi_{\mathrm{loc}}(\partial_\xi X,H_r)$ when $h\to +\infty$.

\medskip
\noindent Define the trace map from $\xi$ by 
$$\mathcal{T}_\xi:\left\{f:X_0\to\R:\ df\in \Lp^\phi_{\mathrm{loc}}(X_1,\xi)\right\}/\R\to L^\phi_{\mathrm{loc}}\left(\partial_\xi X,H_r\right)/\R,\ f\mapsto f_\infty.$$
We can determine the kernel and the image of $\mathcal{T}_\xi$ as we did for $\mathcal{T}$.

\begin{proof}[Proof of Theorem \ref{thelocaltracemap}] We explain how to adapt the proofs of Propositions \ref{nucleo} and \ref{image} given in Section \ref{boundary}. First, notice that for each subcomplex $Y_n$, the intersection $Y_{n}\cap \{b_r\leq 0\}$ has finite diameter. In particular, in order to show that a function $f:X_0\to\R$ is locally $\phi$-integrable, it is enough to bound from above the norm $\|f\|_{\phi,Y_{n}\cap \{b_r\geq 0\}}$ for all $n\in\N$. Since the sequence of coverings
$$\left\{\mho_\xi(x):x\in X_0\cap\{h\leq b_r\leq h+1\}\right\},\ h\in\N,$$
of $\partial_\xi X$, enjoy the same properties (1), (2), and (3), enumerated before the proof of Lemma \ref{bigF} for the coverings $\{\mho(x):x\in \Sigma_r\}_{r\in\N}$, the same proof as in Lemma \ref{bigF} shows that the kernel of $\mathcal{T}_\xi$ is precisely $(\Lp^\phi_{\mathrm{loc}}(X_0)\oplus\R)/\R$. 

\medskip
\noindent An analogous formula to (\ref{invdistance}) holds for the hororadial shadows. More precisely, the density of the measure $\lambda_r$ verifies
\begin{equation}\label{invdistance2}
\frac{1}{\varrho_r(\zeta,\eta)^{2Q}}\asymp \sum_{\sigma\in X_1}\frac{\I_{\mho_r(\sigma_-)}(\zeta)\I_{\mho_r(\sigma_+)}(\eta)}{H_r\left(\mho_r(\sigma_-)\right)H_r\left(\mho_r(\sigma_+)\right)},\text{ for }\zeta,\eta\in\partial_\xi X.
\end{equation}
The proof is a minor adaptation of (\ref{invdistance}). Therefore, the same arguments as in Proposition \ref{image} show that the image of $\mathcal{T}_\xi$ is the local Orlicz-Besov space $B^\phi_{\mathrm{loc}}(\partial_\xi X,\varrho_r)/\R$.
\end{proof}

\begin{remark} Consider the continuous local $\Lp^\phi$-cohomology
$$A^\phi_{\xi}=A^\phi_{\mathrm{loc}}\left(\partial_\xi X,\varrho_r\right):=\left\{u\in B^\phi_{\mathrm{loc}}\left(\partial_\xi X,\varrho\right): u\text{ is continuous}\right\}.$$
The space $A^\phi_{\mathrm{loc}}\left(\partial_\xi X,\varrho_r\right)$ is a unital commutative Fr\'echet algebra when equipped with the countable family of multiplicative seminorms
$$p_n(u):=\|u\|_{\infty,K_n}+\langle u\rangle_{\phi,K_n},\ n\in\N.$$
Its spectrum $\mathrm{Sp}\left(A^\phi_\xi\right)$ is a Hausdorff hemicompact topological space invariant by Fr\'echet algebra isomorphisms. We refer to \cite{HelmutGoldmann90} for background on Fr\'echet algebras. It is homeomorphic to the quotient $\partial_\xi X/\sim_{\phi}$, where the local $\Lp^\phi$-cohomology classes are defined by setting $\zeta\sim_\phi \zeta'\in \partial_\xi$ if, and only if, $u(\zeta)=u(\zeta')$ for any function $u\in A^\phi_\xi$. These classes must be preserved, in particular, by the boundary homeomorphism of a self quasi-isometry of $X$ which fixes $\xi$.
\end{remark}

\noindent The next lemma is relevant to Corollary \ref{confdimheintze}.

\begin{lemma}\label{localLip}
Let $\phi$ be a the doubling Young function such that
$$\sum_{j=0}^{+\infty}\phi(2^{-j})2^{jQ}<\infty.$$
Suppose that $(\partial X,\varrho)$ is $Q$-regular and let $\xi\in \partial X$. Then $A^{\phi}_{\mathrm{loc}}(\partial_\xi X,\varrho_r)$ contains all Lipschitz functions.
\end{lemma}
\noindent This is the case for $\phi_{p,\kappa}$, the doubling Young defined in (\ref{ourphi}), whenever $(p,\kappa)>(Q,1)$.
\begin{proof}
Let $u:\partial_\xi X\to\R$ be a $\mathsf{L}$-Lipschitz function and let $K\subset \partial_\xi X$ be a compact set. Denote by $\delta=\diam\,K$, and for $\zeta\in K$ and $j\in\N$, consider the sets $A_j(\zeta)$ defined as $\overline{B}(\zeta,2^{-j}\delta)\setminus B(\zeta,2^{-(j+1)}\delta)$. Then
\begin{align*}
\sum_{j=0}^{+\infty}\int_{A_j(\zeta)}\frac{\phi\left(u(\eta)-u(\zeta)\right)}{\varrho_r(\eta,\zeta)^{2Q}}dH_r(\eta)&\leq \sum_{j=0}^{+\infty}\int_{A_j(\zeta)}\frac{\phi\left(\mathsf{L}\varrho_r(\eta,\zeta)\right)}{\varrho_r(\eta,\zeta)^{2Q}}dH_r(\eta)\\
&\lesssim \sum_{j=0}^{+\infty}\phi(2^{-j})2^{jQ}<\infty.
\end{align*}
Integrating over $K$ we obtain the result.
\end{proof}

\section{Quasi-isometries of Heintze groups}\label{hnc}

\noindent In this section we give the proofs of Theorems \ref{fixpoint}, \ref{localclassesinfty}, \ref{localclassesN}, and Corollaries \ref{biLip} and \ref{abelian}.

\subsection{Preliminaries and notation} Let $X_\alpha=N\rtimes_\alpha \R$ be a purely real Heintze group and $\mathfrak{x}_\alpha=T_{(e,0)}X_\alpha$ be its Lie algebra. The left translations are then given by 
$$d_{(e,0)}L_{(x,t)}=\left(d_eL_x\circ\mathrm{Exp}(-t\alpha),1\right), \ (x,t)\in X_\alpha.$$ 
We will identify the subgroup $N\times \{0\}$ with $N$. Let $\mathcal{B}=\{\partial_1,\ldots,\partial_n\}$ be a basis of $T_eN$ on which $\alpha$ assumes its Jordan canonical form, and denote by $\partial_t$ the tangent vector at $s=0$ to the curve $s\mapsto (e,s)$.

\medskip
\noindent Let us fix and describe a negatively curved metric on $X_\alpha$. Given $\lambda>0$, denote by $\alpha_\lambda$ the derivation $\lambda\alpha$. Notice that $X_{\alpha_\lambda}$ is always isomorphic to $X_\alpha$. By \cite[Theorem 2]{Heintze74}, if $g_N$ is a left invariant metric on $N$ for which the symmetric part of $\alpha$ is positive definite and $\lambda$ is large enough, then
\begin{equation}\label{fixmetric}g_{(x,t)}=\tau_\lambda(t)^*g_{N,x}\oplus dt^2\end{equation}
defines a metric on $X_{\alpha_\lambda}$ of negative sectional curvature bounded from above by $-1$. By \cite[Theorem 3]{Heintze74}, if $b_1,\ldots,b_{n}$ are well chosen positive numbers, then the symmetric part of $\alpha$ for the left invariant metric on $N$ which makes $\{b_1\partial_1,\ldots,b_n\partial_n\}$ an orthonormal basis, is positive definite. Here the $b_i$ depend only on the dimension $n$ and the eigenvalues of $\alpha$. Therefore, $X_{\alpha_\lambda}$ is a $\mathrm{CAT}(-1)$ space when equipped with (\ref{fixmetric}).

\medskip
\noindent The vertical lines $\gamma_x:t\mapsto (x,t)$ are unit-speed geodesics in $X_{\alpha_\lambda}$, and define the boundary point denoted $\infty$. The boundary at infinity of $X_{\alpha_\lambda}$ is a topological $n$-sphere, which we identify with the one-point compactification $N\cup\{\infty\}$. The orbits of $N$, the sets $N\times\{t\}$, correspond to the horospheres centered at $\infty$. We denote the Riemannian norm and distance induced in $N\times \{t\}$ by $\|\cdot\|_t$ and $d_t$ respectively. The metric on $X_{\alpha_\lambda}$ is then given by
\begin{equation}\label{metric}
\|(v,s)\|^2_{(x,t)}=\|\mathrm{Exp}(t\alpha_\lambda)\circ d_xL_{x^{-1}}(v)\|^2_0+s^2,\ (v,s)\in T_{(x,t)}X_{\alpha_\lambda}.
\end{equation}
The volume element is $dV(x,t)=e^{\lambda\mu t}dV_N(x)dt,$ where $dV_N$ is the volume element of the corresponding left-invariant metric on $N$, and $\mu:=\mathrm{tr}(\alpha)$ is the trace of $\alpha$.

\medskip
\noindent In $N=\partial_\infty X_{\alpha_\lambda}$, the parabolic visual metric is comparable to the function defined as
\begin{equation}\label{quasimetric}
\varrho_\infty(x,y)\asymp e^{-t},\text{ where }t=\sup\left\{s\in\R:d_s\left(\gamma_x(s),\gamma_y(s)\right)\leq 1\right\},\ x,y\in N.
\end{equation} 
The assertion follows form the fact that both functions are invariant by left translations in $N$, and admit the one-parameter group of contractions $\{\tau_\lambda(t):t\in\R\}$, which contract by the factor $e^{-t}$. We refer to \cite{Hersonsky-Paulin97,Hamenstadt89} for more details. This also implies that the parabolic boundary $(N,\varrho_\infty,dV_N)$ is Ahlfors regular of dimension $\lambda\mu$.

\medskip
\noindent 
Recall that $\mu_1\leq\cdots\leq\mu_d$ are the distinct eigenvalues of $\alpha$, and that $\alpha$ in the generalized eigenspace $V_i$ is a sum of $k_i$ Jordan blocks of sizes $m_{ij}$, for $j=1,\ldots,k_i$. Write the vectors of $\mathcal{B}$ as $\{\partial_k^{ij},\ (i,j,k)\in\mathcal{I}\},$ where $\mathcal{I}$ is the set of triples
$$\mathcal{I}:=\left\{(i,j,k):i\in\{1,\ldots,d\},j\in\{1,\ldots,k_i\},k\in\{1,\ldots,m_{ij}\}\right\}$$
equipped with the lexicographic order.

\subsection{The singular value decomposition of $\mathrm{Exp}(tJ)$}

\noindent In this section we prove a lemma about the asymptotic behaviour when $t\to+\infty$ of the singular value decomposition of the matrix $M=\mathrm{Exp}(tJ)$, where $J$ is the standard nilpotent matrix of size $m$:
$$J=J(0,m)=\left(
\begin{array}{c c c c c}
0 & 1& 0 & \cdots & 0\\
0 & 0 & 1& \cdots & 0\\
\vdots&\vdots&\vdots& \ddots& \vdots\\
0 & 0& 0 & \cdots &1\\
0&0&0&\cdots & 0
\end{array}\right).$$
Since any two inner products on $\R^m$ are equivalent, for our purposes it suffices to consider the standard one $\langle\cdot,\cdot\rangle$, which makes the standard basis $\{e_i\}_{i=1}^n$ orthonormal. The next lemma provide us with the key estimates needed to compute the $L^{p,\kappa}$-cohomology of $X_\alpha$.

\begin{lemma}\label{bon}
Let $M=\mathrm{Exp}(tJ)$ be as above. Let $\{v_1,\ldots,v_m\}$ be an orthonormal basis of eigenvectors for $M^*M$. Denote $\lambda_i$ the eigenvalue associated with $v_i$, and suppose they are ordered $\lambda_1<\lambda_2<\ldots<\lambda_m$. Then the following properties are satisfied:
\medskip
\begin{enumerate}
\item For each $i=1,\ldots, m$, $\lambda_i\lambda_{m-i+1}=1$. In particular, $\lambda_{\frac{m+1}{2}}=1$ if $m$ is odd.
\medskip
\item For each $i=1,\ldots,m$,
\begin{equation}\label{eigenvalues}
\sqrt{\lambda_i}\sim \frac{(m-i)!}{(i-1)!}\,t^{2i-m-1}, \text{ when }t\to +\infty.
\end{equation}
\item For each $i=1,\ldots,m$, the $\lambda_i$-eigenspace $\R v_i$ of $M^*M$ tends to $\R e_i$ when $t\to +\infty$.
\end{enumerate}
\end{lemma}
\begin{proof}
To prove (1) note that $M^{-1}$ is similar to $M$. More precisely, denote by $R$ the reflexion given by $R(e_i)=(-1)^ie_i$, then $RMR=M^{-1}$. Therefore $M^{-1}$ and $M$ have the same eigenvalues. This proves (1).

\noindent We now prove (2). Let $S$ be the reflexion given by $S(e_i)=e_{m-i+1}$. Then $SMS=M^*$, and therefore, $M^*M=(SM)^2$. This implies that $M^*M$ and $SM$ have the same eigenvectors, and that $\beta_i=\pm\sqrt{\lambda_i}$, $i=1,\ldots,m$, are the eigenvalues of $SM$. Also note that $SM$ is symmetric.

\noindent Denote by $T_k$ the trace of $(SM)^k$, and let 
$$p(x)=\det(xI-SM)=x^m+c_1x^{m-1}+\cdot+c_{m-1}x+c_m,$$
be the characteristic polynomial of $SM$. Using the Newton identities
$$c_1=-T_1\text{ and } T_k+c_1T_{k-1}+\cdots+c_{k-1}T_1 +kc_k=0,\ k=2,\ldots,m,$$
one shows by induction that 
\begin{equation}\label{newton}
|c_k|\sim \frac{(k-1)!(k-2)!\cdots 1}{(m-k)!(m-k+1)!\cdots (m-1)!}\,t^{k(m-k)}\text{ when } t\to+\infty.
\end{equation}
On the other hand, by Vieta's identities we have
$$c_k=(-1)^k\sum\limits_{1\leq i_1<\cdots<i_k\leq m}\beta_{i_1}\cdots\beta_{i_k}\sim (-1)^k\beta_{m-k+1}\cdots\beta_{m-1}\beta_m,$$
which together with (\ref{newton}) gives (2).

\noindent Let us prove (3). Since $SM$ is symmetric, we have that for each $j=1,\ldots,m$, the eigenvector $v_j$ is orthogonal to the image of $SM-\beta_jI$. Denote by 
$$w_i=SM(e_i)=\sum_{k=m-i+1}^{m}\frac{t^{k-m+i-1}}{(k-m+i-1)!}e_k,\text{ for }i=1,\ldots,m.$$
Note that since $RMR=M^{-1}$, it is enough to prove the assertion for $j>(m+1)/2$. Fix such a $j$, so that $|\beta_j|\to +\infty$ when $t\to +\infty$. In this case $1<2j-m\leq m$. The highest degree of the coefficients of $w_i$ is $i-1$, given by the coefficient in $e_m$. In particular,
$$u_i=\frac{w_i-\beta_je_i}{\|w_i-\beta_je_i\|}\to \pm e_i\text{ when }t\to+\infty,$$
for $i<2j-m$. Therefore, any limit point of $v_j$ is orthogonal to $\text{span}(e_1,\ldots,e_{2j-m-1})$. In particular, the limit points of $v_m$ are $\pm e_m$.

\noindent We consider now the case when $j\leq m-1$ and $i\in \{2j-m,\ldots,j-1\}$. Let $r=m-j-1$, and define recursively
\begin{align*}
z_i^{(0)}&=w_i\\
z_i^{(s)}&=\frac{d_{is}}{t}z_{i+1}^{(s-1)}-z_i^{(s-1)},\ s=1,\ldots,r,
\end{align*}
where $d_{is}$ is chosen so that $z_i(s)$ has zero coefficient in $e_{m-(s-1)}$ (e.g. $d_{i1}=i$). Then the the coefficient of $z_i(r)$ in $e_k$ is zero for $k=m-(r-1),\ldots,m$. Moreover, the highest degree of the coefficients of $z_i(r)$ is $i-(r+1)$, given by the coefficient in $e_{m-r}$.

\medskip
\noindent The recursion defined above but starting at $w_i-\beta_ie_i$ instead of $w_i$ produces a vector $z_i$ in the image of $SM-\beta_iI$, which in the case $i-r<2j-m$ verifies
$$\frac{z_i}{\|z_i\|}\to \pm e_i,\text{ when }t\to +\infty.$$
This shows that the limit points of $v_j$ are orthogonal to $e_i$ for $i\leq 2j-m+r=j-1$. This finishes the proof of (3).
\end{proof}

\noindent Let us apply the previous lemma in our setting. Consider the two inner products on $T_eN$, $\langle\cdot,\cdot\rangle$ and $\langle\cdot,\cdot\rangle_0$, who make $\{\partial_k\}$ and $\{b_k\partial_k\}$ orthonormal basis respectively. Note that $\mathrm{Exp}(t\alpha_\lambda)$ has the diagonal block form
$$\bigoplus_{i=1}^d\bigoplus_{j=1}^{k_i}e^{t\mu_i}\mathrm{Exp}(\lambda tJ(0,m_{ij})).$$
If $v$ is any vector in $V_{ij}=\mathrm{Span}\left(\partial^{ij}_k,\ k=1,\ldots, m_{ij}\right)$, then
$$\left\|\mathrm{Exp}(t\alpha_\lambda)(v)\right\|_0\asymp e^{\lambda t\mu_i}\left\|\mathrm{Exp}(\lambda tJ(0,m_{ij})(v)\right\|.$$
Let $\{v^{ij}_1(t),\ldots,v^{ij}_{m_{ij}}(t)\}$ be an orthonormal basis (for $\langle\cdot,\cdot\rangle$) of $V_{ij}$ formed by singular eigenvectors of $\mathrm{Exp}(\lambda tJ(0,m_{ij}))$, and denote by $\lambda^{ij}_1(t)<\cdots<\lambda^{ij}_{m_{ij}}(t)$ its singular values. By item (2) of Lemma \ref{bon}, we have 
\begin{equation*}
\sqrt{\lambda^{ij}_k(t)}\sim \sf{c}^{ij}_{k}t^{2k-m_{ij}-1},\text{ when } t\to+\infty,
\end{equation*}
where $\sf{c}^{ij}_{k}=\lambda^{2k-m_{ij}-1}(m_{ij}-k)!/(k-1)!$. In particular,
\begin{equation}\label{equivlambda}
\left\|\mathrm{Exp}(t\alpha_\lambda)\left(v^{ij}_k(t)\right)\right\|_0\asymp t^{2k-m_{ij}-1}e^{\lambda t\mu_i},\text{ when }t\to+\infty.
\end{equation}
Moreover, by item (3) of Lemma \ref{bon}, we can chose $v^{ij}_k(t)$ so that
$$\left\|v^{ij}_k(t)-\partial^{ij}_{k}\right\|_0\to 0,\text{ when }t\to+\infty.$$
In other words, the singular eigenvectors converge to the standard basis vectors of $V_{ij}$. 

\subsection{The degree one $L^{p,\kappa}$-cohomology of $X_\alpha$: main computations}

\noindent Consider $u$ a real smooth function on $X_{\alpha_\lambda}$. For each $t\in \R$, denote by $u_t:N\to \R$ the function $u_t(x)=u(x,t)$, and let
$$u_\infty(x):=\lim_{t\to+\infty}u_t(x),$$
whenever this limit exists.  

\medskip
\noindent Let $E\subset N$ be a measurable set and $t_0\in\R$. As we proved in Section \ref{horoshift}, the hororadial shift
$$w\in L^{p,\kappa}\left(E\times [t_0,+\infty)\right)\mapsto S_k(w)(x,t)=\int_{k}^{k+1}w(x,t+s)ds,$$
is a contraction in the $\phi_{p,\kappa}$-norm, and its norm is bounded from above by $e^{-k\lambda\mu/\sf{K}}$, where $\sf{K}$ is the doubling constant of $\phi_{p,\kappa}$. The proof in this case follows from Jensen's inequality. Therefore, the operator
$$T(w)(x,t)=\sum_{k=0}^\infty S_k(w)(x,s)=\int_0^{+\infty}w(x,t+s)ds,$$
is bounded in the $\phi_{p,\kappa}$-norm. 

\medskip
\noindent Suppose that $du\in L^{p,\kappa}\left(E\times [t_0,+\infty)\right)$. Under mild assumptions on $E$, for instance if $E$ is $\lambda\mu$-Ahlfors regular, the functions $u_t$ converge a.e. and in $L^{p,\kappa}_{\mathrm{loc}}(E,dV_N)$ to $u_\infty$. In particular, $u_\infty\in L^1_{\mathrm{loc}}(E,dV_N)$ and $du_\infty$ is well defined in the sense of distributions. Moreover, the operator $T$ evaluated at the function $\partial_tu$ gives
$$T(\partial_tu)(x,t)=u_\infty(x)-u_t(x).$$
Thus, we can write $u_\infty=v+u$, where $v\in L^{p,\kappa}\left(E\times [t_0,+\infty)\right)$.

\medskip
\noindent For $(i,j,k)\in\mathcal{I}$, we denote by $X^{ij}_k$ the left invariant vector field of $N$ generated by the basis vector $\partial^{ij}_k$, and by $Y^{ij}_k(t)$ the left invariant vector field of $N$ generated by $v^{ij}_k(t)$. Recall from the previous section that $v^{ij}_k(t)\in V_{ij}$ is a singular eigenvector of $\mathrm{Exp}(t\alpha_\lambda)$, and $v^{ij}_k(t)\to\partial^{ij}_k$ when $t\to +\infty$. When the pair $(i,j)$ is clear from the context, we simply write $X_k$, $\partial_k$, $Y_k(t)$, $v_k(t)$, and $m$ respectively. Recall that for $i\in\{1,\ldots,d\}$, the $i$-critical exponent is given by $p_i=\mu/\mu_i$.

\begin{lemma}\label{critical}
Let $(p,\kappa)\in I$, and let $E$ denotes either $N$ or $N\setminus\{\xi\}$ for some point $\xi$. Consider $u:X_{\alpha_\lambda}\to \R$ a smooth function such that $du\in L^{p,\kappa}(K\times [t,+\infty))$ for any compact $K\subset E$ and $t\in\R$. Then, $X^{ij}_k(u_\infty)=0$ whenever
\begin{equation}\label{condicion}(p,\kappa)\leq \big(p_i,1+p_i(m_{ij}+1-2k)\big).\end{equation}
\end{lemma}
\begin{proof}
Let $u:X_{\alpha_\lambda}\to \R$ be as in the statement. For each $t\geq 0$, consider the function $Y_k(t)(u_t):E\to \R$. We first show that $Y_k(t)(u_t)\to X_k(u_\infty)$ in the sense of distributions.

\medskip
\noindent Let $\zeta:E\to\R$ be a smooth function with compact support $K$. We first observe that
\begin{align*}
\langle \left(Y_k(t)-X_k\right)(u_t),\zeta\rangle&= -\int_E u_t(x)d_e\left(\zeta\circ L_x\right)(v_k(t)-\partial_k)dV_N(x).
\end{align*}
Moreover, since $e\in \mathrm{supp}\left(\zeta\circ L_x\right)=L_x^{-1}\left(K\right)$ if, and only if, $x\in K$, we have
$$\left|\langle \left(Y_k(t)-X_k\right)(u_t),\zeta\rangle\right|\leq \|u_t\|_{1,K}\max_{x\in K}\left|d_e\left(\zeta\circ L_x\right)\right|_0\|v_k(t)-\partial_k\|_0$$
which tends to zero when $t\to +\infty$ by item (3) of Lemma \ref{bon}.

\medskip
\noindent Finally, we obtain
\begin{align*}
\left|\langle Y_k(t)(u_t),\zeta\rangle-\langle X_k(u_\infty),\zeta\rangle\right|& \leq \left|\langle Y_k(t)(u_t)-X_k(u_t),\zeta\rangle\right|+\left|\langle X_k(u_t)-X_k(u_\infty),\zeta\rangle\right|,
\end{align*}
which tends to zero since $X_k(u_t)\to X_k(u_\infty)$.

\medskip
\noindent We now show that $X_k(u_\infty)=0$. By the definition of $v_k(t)$ and (\ref{equivlambda}), we have
$$\tau_k(t):=\|Y_k(t)(x)\|_t=\|\mathrm{Exp}(t\alpha)(v_k(t))\|_0\asymp t^{2k-m_{ij}-1}e^{\lambda t\mu_i}.$$
Also
$$\left|Y_k(t)(u_t)(x)\right|\leq \tau_k(t)\left|d_{(x,t)}u\right|.$$
Let $\mathsf{c}>0$ be a constant to be determined later. Consider the set 
$$E_t=\left\{x\in E: \left|Y_k(t)(u_t)(x)\right|\geq e^{-\mathsf{c}t}\right\}.$$
Then, for $x\in E_t$ we have
$$\frac{\tau_k(t)^{-p}\left|Y_k(t)(u_t)(x)\right|^p}{\log^\kappa\left(e+\tau_k(t)e^{\mathsf{c}t}\right)}\leq \phi_{p,\kappa}\left(\frac{\left|Y_k(t)(u_t)(x)\right|}{\tau_k(t)}\right)\leq \phi_{p,\kappa}\left(\left|d_{(x,t)}u\right|\right).$$
Moreover, there exists a constant $\sf{C}>0$, and $t_0>0$, such that
$$\frac{1}{\sf{C}}\frac{e^{-p\lambda\mu_i t}\left|Y_k(t)(u_t)(x)\right|^p}{t^{\kappa+p(2k-m_{ij}-1)}}\leq \frac{\tau_k(t)^{-p}\left|Y_k(t)(u_t)(x)\right|^p}{\log^\kappa\left(e+\tau_k(t)e^{\mathsf{c}t}\right)},\text{ for } t\geq t_0.$$
Let $q=\kappa+p(2k-m_{ij}-1)$, then for any compact $K\subset E$, we have
$$A:=\int_{t_0}^{+\infty}\int_{E_t\cap K}\frac{\left|Y_k(t)(u_t)(x)\right|^pe^{\lambda(\mu-p\mu_i) t}}{t^q}dV_N(x)dt\leq \sf{C}\int_{K\times [t_0,+\infty)}\phi_{p,\kappa}\left(|du|\right)dV<\infty.$$
This implies
$$\int_{t_0}^{+\infty}\int_K\frac{\left|Y_k(t)(u_t)(x)\right|^pe^{\lambda(\mu-p\mu_i) t}}{t^q}dV_N(x)dt\leq A+V_N(K)\int_{t_0}^{+\infty}\frac{e^{\left(-p\mathsf{c}+\lambda(\mu-p\mu_i)\right) t}}{t^q}dt.$$
Let $\mathsf{c}>-\lambda p^{-1}(\mu-p\mu_i)$. Therefore, whenever $(p,\kappa)$ satisfies (\ref{condicion}), there is a subsequence $t_l\to+\infty$ such that
$$\left\|Y_k(t_l)(u_{t_l})\right\|_{p,K}\to 0,\text{ when }l\to +\infty.$$
This shows that $X_k(u_{\infty})=0$.
\end{proof}

\medskip
\noindent We denote by $\{\theta^{ij}_k\}$ the basis of $V_{ij}^*$, dual to the standard left invariant basis $\{X_k^{ij}\}$. For each $t\in\R$, we have
\begin{equation}\label{dualnorm}
\left|\theta^{ij}_k\right|_t=\left\|\mathrm{Exp}\left(-t\alpha_\lambda^*\right)(\partial^{ij}_k)\right\|_0\asymp\left[\sum_{r=k}^{m_{ij}}\frac{t^{2(r-k)}}{(r-k)!^2}\right]^{1/2}e^{-\lambda\mu_i t}\asymp t^{m_{ij}-k}e^{-\lambda\mu_i t}.
\end{equation}
In the following lemma, we will use the functions $\sf{i}$ and $\sf{m}$ that assign to a subset $\mathcal{I}_0$ of $\mathcal{I}$ the integers
\begin{align*}
\sf{i}=\sf{i}\left(\mathcal{I}_0\right)&:=\min\left\{r:(r,s,l)\notin \mathcal{I}_0\right\},\\
\sf{m}=\sf{m}\left(\mathcal{I}_0\right)&:=\max\left\{m_{\sf{i}s}-l:(\sf{i},s,l)\notin \mathcal{I}_0\right\}.
\end{align*}

\noindent Given a smooth function $u_\infty:N\to \R$, by its hororadial extension we mean the function independent of $t$, $u:X_{\alpha_\lambda}\to\R$, defined by $u(x,t)=u_\infty(x)$.

\begin{lemma}\label{nonconstant}
Let $p\geq 1$ and $\kappa\geq 0$. Fix $\mathcal{I}_0$ a subset of $\mathcal{I}$ and denote by $\sf{i}=\sf{i}\left(\mathcal{I}_0\right)$ and $\sf{m}=\sf{m}\left(\mathcal{I}_0\right)$. Consider $u_\infty:N\to \R$ a smooth function such that
$$X^{rs}_l(u_\infty)=0\text{ for all }(r,s,l)\in\mathcal{I}_0.$$
Let $u:X_{\alpha_\lambda}\to \R$ be its hororadial extension. Then $du\in L^{p,\kappa}\left(K\times[t_0,+\infty)\right)$ for any compact $K\subset N$ and $t_0\in\R$, whenever
\begin{equation*}
(p,\kappa)>\big(p_{\sf{i}},1+p_{\sf{i}}\sf{m}\big).
\end{equation*}
Moreover, in both cases, if $u_\infty$ is not identically zero and the compact $K\subset N$ has positive measure, then $u\notin L^{p,\kappa}\left(K\times[0,+\infty)\right)$.
\end{lemma}
\begin{proof}
Let $K\subset N$ be a compact set. Since $u$ does not depend on $t$, it is clear that $u\notin L^{p,\kappa}\left(K\times[0,+\infty)\right)$ when $K$ has positive measure, unless $u_\infty$ is identically zero.

\medskip
\noindent If $\mathcal{I}_0=\mathcal{I}$, there is nothing to prove since $u$ is constant. So suppose that $\mathcal{I}_0$ is a proper subset of $\mathcal{I}$. Let $\sf{C}$ be an upper bound in $K$ for $|X^{rs}_l(u_\infty)|$, for any $(r,s,l)\in\mathcal{I}$. Then, by (\ref{dualnorm}), there is $t_1>0$ such that
$$\left|d_{(x,t)}u\right|\leq \sf{C}\sum_{(r,s,l)\notin \mathcal{I}_0}\left|\theta^{rs}_l\right|_t\lesssim \sum_{(r,s,l)\notin\mathcal{I}_0}t^{m_{rs}-l}e^{-\lambda\mu_r t}\lesssim t^\sf{m}e^{-\lambda\mu_{\sf{i}}t},$$ 
for $x\in K$, and $t\geq  t_1.$ Therefore,
$$\int_{t_1}^{+\infty}\int_K\phi_{p,\kappa}\left(|d_{(x,t)}u|\right)e^{\lambda\mu t}dV_N(x)dt\lesssim V_N(K)\int_{t_1}^{+\infty}\frac{e^{\lambda(\mu-p\mu_{\sf{i}})t}}{t^{\kappa-p\sf{m}}}dt.$$
The conclusion follows from this inequality.
\end{proof}

\noindent We will apply the previous lemma later to the following subsets of indices. Let $\mathcal{K}_0=\emptyset$, and for each $i\in \{1,\ldots,d\}$, let
\begin{equation}\label{indexK}\mathcal{K}_i=\{(r,j,k):r\leq i\},\ \mathcal{H}_i=\mathcal{K}_{i-1}\cup\{(i,j,1):m_{ij}=m_i\}.\end{equation}
Then, for $0\leq i<d$ we have 
$$\sf{i}(\mathcal{K}_i)=i+1,\text{ and } \mathsf{m}(\mathcal{K}_i)=m_{i+1}-1.$$
Also, for any $i$, we have 
$$\sf{i}(\mathcal{H}_i)=i,\text{ and }\mathsf{m}(\mathcal{H}_i)=m_i-2,$$
if $m_i\geq 2$. Notice that if $m_i=1$, then $\mathcal{K}_i=\mathcal{H}_i$.

\medskip
\noindent For $\xi\in N$, $r>0$ and $t\in \R$, we denote by $Y(r,t)$ the complement in $X_{\alpha_\lambda}$ of the set $B(\xi,r)\times (t,+\infty)$, where $B(\xi,r)$ is the open ball of radius $r$ centered at $\xi$ in $N$.

\medskip
\noindent The proof of the next lemma is inspired by \cite[Lemma 20]{PierrePansu07}.

\begin{lemma}\label{infinity}
Let $\xi\in N$, $(p,\kappa)\in I$, and let $u:X_{\alpha_\lambda}\to \R$ be a smooth function such that $du\in L^{p,\kappa}\left(Y(r,t)\right)$, for all $r>0$ and $t\in \R$. Suppose that there is a left-invariant vector field $X$ of $N$, generated by an eigenvector $X_e$ of $\alpha$, such that $X(u_\infty)=0$. If $Z$ is a left-invariant vector field of $N$ such that $[X,Z]=0$, then $Z(u_\infty)=0$.
\end{lemma}
\begin{proof}
We denote by $\tilde{X}$ and $\tilde{Z}$ the left-invariant vector fields of $X_{\alpha_\lambda}$ generated by $(X_e,0)$ and $(Z_e,0)$ respectively. Notice that they coincide with $(X,0)$ and $(Z,0)$ in $N$. Let $\tilde{u}_\infty$ be the hororadial extension of $u_\infty$. Since $X_e$ is an eigenvector of $\alpha$, we have $\tilde{X}(\tilde{u}_\infty)=0$. Moreover, $\tilde{X}$ and $\tilde{Z}$ also commute.

\medskip
\noindent First, notice that since the exponential map of $N$ is a diffeomorphism, the one-parameter subgroup $\{\exp(s (X_e,0)):s\in\R\}$ is closed in $X_{\alpha_\lambda}$. This implies that the flow of $\tilde{X}$,
$$\varphi_s(x,t)=(x,t)\cdot\exp(s (X_e,0)),$$
satisfies the following property: for any compact subset $K\subset X_{\alpha_\lambda}$, there exists $l_K\geq 0$ such that if $|s-s'|\geq l_k$, then $\varphi_{s}(K)\cap\varphi_{s'}(K)=\emptyset$. 

\medskip
\noindent Moreover, since $N$ is nilpotent and the coefficient of $[\tilde{X},\partial_t]$ in $\partial_t$ is zero, we see that $\mathrm{tr}\left(\mathrm{ad}_{\tilde{X}}\right)=0$. In particular,
$$\varphi_s^*(dV)=e^{-\mathrm{tr}\left(\mathrm{ad}_{\tilde{X}}\right)}dV=dV,$$
i.e. the flow $\varphi_s$ preserves the measure $dV$.

\medskip
\noindent Suppose that $w:X_{\alpha_\lambda}\to\R$ is a function in $L^{p,\kappa}\left(Y(r,t)\right)$ for all $r$ and $t$. We claim that for each compact set $K\subset X_{\alpha_\lambda}$, we have
$$\left\|w\circ \varphi_s\right\|_{1,K}\to 0, \text{ when }s\to +\infty.$$
Indeed, fix such a compact set $K$, and choose $r$ and $t$ so that $\varphi_s(K)\subset Y(r,t)$ for all $s\in\R_+$. Given any sequence $s_l\to +\infty$, we can extract a subsequence $s_{l'}\to +\infty$ such that the images $\{\varphi_{s_{l'}}(K):l'\in\N\}$ are pairwise disjoint. Therefore,
$$\sum_{l'\in\N}\int_K\phi_{p,\kappa}\left(w\circ\varphi_{s_{l'}}\right)dV=\sum_{l'\in\N}\int_{\varphi_{s_{l'}}(K)}\phi_{p,\kappa}\left(w\right)dV\leq \int_{Y(r,t)}\phi_{p,\kappa}(w)dV<\infty.$$
This implies that $\left\|w\circ\varphi_{s_{l'}}\right\|_{p,\kappa,K}\to 0$ when $l'\to +\infty$. Since $L^{p,\kappa}(K)\subset L^1(K)$ and the inclusion is norm continuous, the claim follows.

\medskip
\noindent Write $\tilde{u}_\infty=w+u$, where $w\in L^{p,\kappa}(Y(r,t))$, for all $r$ and $t$. From the claim above, we have that $w\circ \phi_s\to 0$, when $s\to +\infty$, in the sense of distributions. Also, since $du\in L^{p,\kappa}(Y(r,t))$ for all $r$ and $t$, and $\tilde{Z}$ has constant norm, the same is true for the function $\tilde{Z}(u)$

\medskip
\noindent Again by the claim above, we have $\tilde{Z}(u)\circ \varphi_s\to 0$, when $s\to+\infty$, in the sense of distributions. But since $\tilde{Z}$ commutes with $\tilde{X}$, we have
$$\tilde{Z}(u)\circ \varphi_s=\tilde{Z}\left(u\circ\varphi_s\right), \text{ for all }s\in\R.$$
Therefore,
$$\tilde{Z}\left(\tilde{u}_\infty\right)=\tilde{Z}\left(\tilde{u}_\infty\circ\varphi_s\right)=\tilde{Z}\left(w\circ\phi_s\right)+\tilde{Z}\left(u\circ\varphi_s\right)\to 0,$$
in the sense of distributions when $s\to+\infty$. That is, $\tilde{Z}\left(\tilde{u}_\infty\right)=0$.

\medskip
\noindent Let $\zeta:N\setminus\{x_0\}\to\R$ and $\eta:\R\to\R$ be smooth functions with compact support. Let $\zeta\times\eta(x,t)=\zeta(x)\eta(t)$. Notice that $\tilde{Z}(x,t)=(Z_t(x),0)$ where $Z_t$ is a left-invariant vector field in $N$. Since $\eta$ does not depend on $x\in N$, we have $Z_t(\eta)=0$. Then
$$0=\langle \tilde{Z}\left(\tilde{u}_\infty\right),\zeta\times\eta\rangle=-
%\int_\R\left[\int_N u_\infty(x)Z_t(\zeta)(x)dV_N(x)\right]\eta(t)e^{\lambda\mu t}dt=
\int_\R\langle u_\infty,Z_t(\zeta)\rangle\eta(t)e^{\lambda\mu t}dt.$$
Since $\zeta$ and $\eta$ where arbitrary, this shows that $Z(u_\infty)=0$.
\end{proof}

\begin{corollary}\label{constantctp}
Let $\xi\in N$, $(p,\kappa)\in I$, and let $u:X_{\alpha_\lambda}\to \R$ be a smooth function satisfying the hypotheses of Lemma \ref{infinity}. Then $u_\infty:N\setminus\{\xi\}\to\R$ is constant a.e..
\end{corollary}
\begin{proof}
Let $Z\neq 0$ be a left-invariant vector field in the center of $\mathfrak{n}$, which exists because $N$ is nilpotent. By Lemma \ref{infinity}, $Z(u_\infty)=0$, that is $u_\infty\circ\varphi_s=u_\infty$ a.e., where $\varphi_s(x)=R_{\exp(sZ)}(x)=L_{\exp(sZ)}(x)$. Using again the fact that the action $s\mapsto\varphi_s$ is proper, we conclude that the Orlicz-Besov norm of $u_\infty$ is zero on any compact set of $N\setminus\{\xi\}$. This proves the lemma.
\end{proof}

\subsection{Proofs of main results}

We now apply the results obtained in the previous section to compute the local and global cohomology of $X_\alpha$.

\medskip
\noindent Recall that for each $i\in\{1,\ldots,d\}$, we have defined $m_i=\max\{m_{ij}:1\leq j\leq k_i\}$, and the subspaces $W_0=\{0\}$, $W_i=V_1\oplus\cdots\oplus V_i$, and 
$$\mathfrak{k}_i=\mathrm{LieSpan}\left(W_i\right),\ V_i^0=\mathrm{Span}\left(\partial^{ij}_1:m_{ij}=m_i\right),\ \mathfrak{h}_i= \mathrm{LieSpan}\left(W_{i-1}\oplus V_i^0\right).$$
We denote also by $K_i$ and $H_i$ the closed Lie subgroups of $N$ whose Lie algebra are $\mathfrak{k}_i$ and $\mathfrak{h}_i$ respectively. Note that the set of left cosets $N/K_i$ are $N/H_i$ are a smooth manifold, and the canonical projections are smooth maps.

\begin{proof}[Proof of Theorem \ref{localclassesinfty}]
First, suppose that $p\in\left(p_{i+1},p_i\right)$ and $\kappa\geq 0$. Let $u:X_\alpha\to\R$ be a smooth function with $du\in L^{p,\kappa}\left(K\times [t_0,+\infty)\right)$, for any compact set $K\subset N$ and $t_0\in\R$. If $i=d$, there is nothing to prove, so suppose $i<d$. Consider the subset $\mathcal{K}_i$ defined in (\ref{indexK}). By Lemma \ref{critical}, $X^{rj}_k(u_\infty)=0$ for all $(r,j,k)\in \mathcal{K}_i$, and thus $X(u_\infty)=0$ for any left-invariant field $X\in \mathfrak{k}_i$. If we suppose $u_\infty$ continuous, this implies that $u_\infty$ is constant on any left coset of $K_i$, and can therefore be regarded as a continuous function on $N/K_i$.

\medskip
\noindent  Note that $\sf{i}(\mathcal{K}_i)=i+1$ and $p>p_{i+1}$. By Lemma \ref{nonconstant}, any smooth function $u_\infty$ on $N/H_i$ can be identified with the hororadial limit of a smooth function $u:X_\alpha\to\R$ with $du\in L^{p,\kappa}\left(K\times [t_0,+\infty)\right)$, for any compact $K\subset N$ and $t_0\in\R$. This shows \emph{(1)}.

\medskip
\noindent Let us prove \emph{(2)}. Suppose that $p=p_i$ and $m_i=1$. Let $\kappa\geq0$. By point \emph{(1)}, if $u:X_\alpha\to\R$ is a smooth function with $du\in L^{p,\kappa}\left(K\times [t_0,+\infty)\right)$, for any compact set $K\subset N$ and $t_0\in\R$, then $X(u_\infty)=0$ for any $X\in \mathfrak{k}_{i-1}$. 

\medskip
\noindent If $\kappa\leq 1$, because $m_i=1$, Lemma \ref{critical} implies that $X(u_\infty)=0$ for any $X\in \mathfrak{k}_i$. If $i=d$, there is nothing more to prove. If not, since $p>p_{i+1}$, the same argument as in \emph{(1)} concludes the proof of the first part of \emph{(2)}.

\medskip
\noindent If $\kappa>1$, because $\sf{i}(\mathcal{K}_{i-1})=i$ and $\sf{m}(\mathcal{K}_{i-1})=m_i-1=0$, Lemma \ref{nonconstant} implies that any smooth function $u_\infty$ on $N/K_{i-1}$ can be identified with the hororadial limit of a smooth function $u:X_\alpha\to\R$ with $du\in L^{p,\kappa}\left(K\times [t_0,+\infty)\right)$, for any compact $K\subset N$ and $t_0\in\R$. This completes the proof of \emph{(2)}.

\medskip
\noindent Suppose now that $p=p_i$ and $1+p_i(m_i-2)<\kappa\leq 1+p_i(m_i-1)$, with $m_i\geq 2$. Consider the subset $\mathcal{H}_i$ defined in (\ref{indexK}). Let $u:X_\alpha\to\R$ be a smooth function with $du\in L^{p,\kappa}\left(K\times [t_0,+\infty)\right)$, for any compact set $K\subset N$ and $t_0\in\R$. By Lemma \ref{critical}, $X^{rj}_k(u_\infty)=0$ for all $(r,j,k)\in\mathcal{H}_i$. Thus $X(u_\infty)=0$ for any left-invariant field $X\in \mathfrak{h}_i$. If we suppose $u_\infty$ continuous, this implies that $u_\infty$ is constant on any left coset of $H_i$, and can be regarded as a function on $N/H_i$.

\medskip
\noindent Since $\sf{i}(\mathcal{H}_i)=i$, $\sf{m}(\mathcal{H}_i)=m_i-2$, and $\kappa>1+p_{i}(m_i-2)$, by Lemma \ref{nonconstant}, any smooth function $u_\infty$ on $N/H_i$ can be identified with the hororadial limit of a smooth function $u:X_\alpha\to\R$ with $du\in L^{p,\kappa}\left(K\times [t_0,+\infty)\right)$, for any compact $K\subset N$ and $t_0\in\R$. This shows \emph{(3)}.
\end{proof}

\begin{proof}[Proof of Theorem \ref{localclassesN}]
Let $[u]\in L^{p,\kappa}H^1(X_\alpha,\xi)$ be a local cohomology class represented by a smooth function $u:X_\alpha\to\R$. By Lemma \ref{critical}, if 
$$(p,\kappa)\leq \left(p_1,1+p_1(m_1-1)\right),$$ 
then $X^{1j}_1(u_\infty)=0$, for any $j$ with $m_{1j}=m_1$. Since $X^{1j}_1$ is a left-invariant vector field of $N$ generated by an eigenvector of $\alpha$, we can apply Corollary \ref{constantctp} to conclude that $u_\infty$ is constant a.e.. Then $[u]=0$.

\medskip
\noindent Suppose that $(p,\kappa)>\left(p_1,1+p_1(m_1-1)\right)$. Consider $u_\infty$ any smooth function with compact support in $N$. Take $\eta:\R\to\R$ a smooth function such that $\eta(t)=0$ for $t<0$ and $\eta(t)=1$ for $t\geq 1$. Set $u:X_\alpha\to\R$ as $u(x,t)=\eta(t)u_\infty(x)$. Then the hororadial limit of $u$ is $u_\infty$, and the same computations as in Lemma \ref{nonconstant} show that $du\in L^{p,\kappa}(X_\alpha)$. Since $[u]\neq 0$, this finishes the proof.
\end{proof}

\begin{proof}[Proof of Theorem \ref{fixpoint}]
Suppose $X_\alpha$ is not of Carnot type. Then, either $m_1=1$, so $\mathfrak{h}_1$ and $\mathfrak{k}_1$ coincide precisely with the proper subalgebra of $\mathfrak{n}$ spanned by the $\mu_1$-eigenvectors of $\alpha$; or $m_1\geq 2$ and $\mathfrak{h}_1$ is a proper Lie subalgebra of $\mathfrak{n}$.

\medskip
\noindent In the first case, we necessarily have $d\geq 2$, and $i(\alpha)=\min\{i:K_i=N\}\geq 2$. In both cases, by Theorems \ref{localclassesinfty} and \ref{localclassesN}, we have
$$p_{\neq 0}(X_\alpha,\infty,I)\leq\left(p_{i(\alpha)},1+p_{i(\alpha)}(m_{i(\alpha)}-1)^+\right)<(p_1,m_1-1)=p_{\neq 0}(X_\alpha,\xi,I),$$
for any $\xi\in N$. This shows that $\infty$ is fixed by the boundary homeomorphism of any quasi-isometry of $X_\alpha$, and that it preserves the left cosets of $H_1$.
\end{proof}

\subsection{Proof of Corollary \ref{biLip}}

Let $N$ be of nilpotency class $\ell$, and let $F\in \mathrm{QIsom}(X_\alpha)$. We must prove that the boundary extension of $F$ is a bi-Lipschitz homeomorphism. 

\medskip
\noindent The derivation $\alpha$ induces a Carnot structure on the subgroup $H_1$, and therefore, the restriction of the snow-flaked parabolic visual metric $\varrho_\infty^{1/\mu_1}$ is bi-Lipschitz equivalent to the Carnto-Carath\'eodory distance of $H_1$, which is, in particular, a geodesic distance. Since $F$ preserves the left cosets of $H_1$, by \cite[Theorem 1.1]{LeDonneXie14}, it is enough to prove:

\begin{enumerate}
\item Any left coset is accumulated by parallel lefts cosets: for any $x\in N$, there exists a sequence $x_j\to x$, $x_jH_1\neq xH_2$, such that
$$\mathrm{dist}\left(x_jH_1,xH_1\right)=\mathrm{dist}(x'_j,xH_1)=\mathrm{dist}(x_jH_1,x'),$$
for all $x'_j\in x_jH$ and $x'\in xH$.
\item Left cosets diverge sub-linearly: for any $x,y\in N$, there exists a sequence $x_j\in xH_1$ such 
$$\frac{\mathrm{dist}\left(x_j,yH_1\right)}{\varrho_\infty(x_j,x)}\to 0,\text{ when }j\to +\infty.$$
\end{enumerate}
To prove (1), let $z_j$ be a sequence in the normalizer of $H_1$, not in $H_1$, and such that $z_j\to e$. Then $x_j=xz_j$ satisfies (1) because $\varrho_\infty$ is left-invariant.

\medskip
\noindent Let us prove (2). Let $x,y\in N$ and write $x^{-1}yx=\exp(Y)$, with $Y\in \mathfrak{n}$. For $t\geq 0$, let $x_t=x\exp(tX)$ with $X\in\mathfrak{h}_1$ a $\mu_1$-eigenvector of $\alpha$. Then $\varrho_\infty(x_t,yx_t)=\varrho_\infty(e,x_t^{-1}yx_t)$. By the Baker-Campbell-Hausdorff formula,
$$x_t^{-1}yx_t=\exp\left(Y'+p(Y,tX)\right),$$
where $Y'$ does not depend on $t$ and $p$ is a polynomial of degree $\ell-1$. Without loss of generality, we can suppose that $\varrho_\infty(e,x_t^{-1}yx_t)\to +\infty$ when $t\to +\infty$.

\medskip
\noindent By definition, $\varrho_\infty(e,z)=e^{s}$ if, and only if, $d_0(e,\tau(s)(z))=1$, where we recall that $d_0$ is the left-invariant Riemannian distance on $N$. Notice that $d_0(e,\exp(Z))\leq \|Z\|_0$ for any $Z\in\mathfrak{n}$. Moreover, there is a constant $\delta>0$ such that 
$$\delta\leq d_0(e,\exp(Z))\leq 1\text{ whenever }\|Z\|_0=1.$$
First, observe that
$$\mathrm{Exp}(-s\alpha)\left(\left[\partial^{ij}_k,\partial^{rs}_l\right]\right)=e^{-(\mu_i+\mu_r)s}\sum_{h=1,h'=1}^{k,l}\frac{(-s)^{m_{ij}+m_{rs}-k-l}}{(m_{ij}-k)!(m_{rs}-l)!}\left[\partial_k^{ij},\partial_l^{rs}\right].$$
This implies, by letting $Z_s:=\mathrm{Exp}(-s\alpha)\left(p(Y,tX)\right)$, that
$$\left\|Z_s\right\|_0\leq \sum_{h=1}^{\ell-1}t^he^{-(h+1)\mu_1 s}q_h(s)\lesssim \sum_{h=1}^{\ell-1}t^he^{-(h+1/2)\mu_1 s},$$
where $q_h(s)$ are polynomials in $s$, and the last inequality holds for $s$ large enough. Choosing $s=s_t$ so that $d_0(e,\exp(tZ_s))=1$, we obtain from the last inequality
$$s_t\leq \frac{\ell}{(\ell+1/2)\mu_1}\log\, t+O(1).$$
Let $X_s=\mathrm{Exp}(-s\alpha)(X)$. A similar, but simpler, computation shows that for 
$$s^*=s^*_t:=\frac{1}{\mu_1}\log\,t+O(1),$$
we have $\delta\leq d_0(e,\exp(tX_{s^*}))\leq 1$. Putting everything together, we obtain
$$\frac{d(x_t,yx_t)}{d(x_t,x)}\lesssim e^{s_t-s_t^*}\to 0,$$
when $t\to+\infty$. This concludes the proof of (2).

\subsection{The abelian case: proof of Corollary \ref{abelian}}

Suppose that $N\simeq \R^n$ is abelian, and let $X_\alpha$ and $X_\beta$ be two quasi-isometric purely real Heintze groups. Notice that being of Carnot type in this case is equivalent to the derivation being a scalar multiple of the identity. Therefore, we can assume that neither $X_\alpha$ nor $X_\beta$ is of Carnot type. We will use the same notations as before, but we add a superscript $\alpha$ or $\beta$ to indicate to which derivation it corresponds.

\medskip
\noindent Write $\alpha$ and $\beta$ in their Jordan form
$$\alpha=\bigoplus_{i=1}^{d^\alpha}\bigoplus_{j=1}^{k^\alpha_i}J(\mu^\alpha_i,m^\alpha_{ij}),\ \beta=\bigoplus_{i=1}^{d^\beta}\bigoplus_{j=1}^{k^\beta_i}J(\mu^\beta_i,m^\beta_{ij}).$$
Consider the functions $s^\alpha$ and $s^\beta$ defined on the index set $\{(p,\kappa):p\geq 1,\kappa\geq 0\}$ which give the topological dimension of the spectrum of the local cohomology with respect to $\infty$.

\medskip
\noindent Since $N$ is abelian, we have $\mathfrak{k}^\alpha_i=W^\alpha_i$ and $\mathfrak{k}^\beta_i=W^\beta_i$. By Theorem \ref{localclassesinfty}, we know that
$$s^\alpha(p,\kappa)=n-\sum_{r=1}^i\mathrm{dim}V^\alpha_r\text{ and }s^\beta(p,\kappa)=n-\sum_{r=1}^i\mathrm{dim}V^\beta_r,$$
for $p\in (p_{i+1}^\alpha,p_i^\alpha)$ and $p\in (p_{i+1}^\beta,p_i^\beta)$ respectively. Note that at each critical exponent $s_\alpha$ and $s_\beta$ jump and change their values. Since $s_\alpha=s_\beta$, we conclude that $d^\alpha=d^\beta=d$, and $\mathrm{dim}V_i^\alpha=\mathrm{dim}V_i^\beta$ for all $i=1,\ldots,d$. In particular,
$$\frac{\mu_i^\alpha}{\mu_i^\beta}=\frac{\mathrm{tr}(\alpha)}{\mathrm{tr}(\beta)}=\lambda>0,\ i=1,\ldots,d.$$
On the other hand, 
$$\mathrm{dim}V_i^\alpha=\sum_{j=1}^{k^\alpha_i}m^\alpha_{ij},\text{ and }\mathrm{dim}V_i^\beta=\sum_{j=1}^{k^\beta_i}m^\beta_{ij}.$$
Moreover, by Theorem \ref{localclassesinfty}, we already know that for each $i\in \{1,\ldots,d\}$,
$$m^\alpha_i=\max\{m^\alpha_{ij}:j=1,\dots,k_i^\alpha\}=m^\beta_i=\max\{m^\beta_{ij}:j=1,\dots,k_i^\beta\}=m_i.$$
In particular, if $m_i=1$ for all $i$ (that is, $\alpha$ and $\beta$ are diagonalizable), then $k_i^\alpha=k_i^\beta$ also, and the proof is completed in this case. We will proceed by induction on $d$ and the size of the biggest Jordan block of the first eigenvalue. 

\medskip
\noindent Suppose that for $i=2,\ldots,d$ we have $k_i^\alpha=k_i^\beta=k_i$ and $m^\alpha_{ij}=m_{ij}^\beta$ for $j=1,\ldots,k_i$. If $d=1$ the assumption is empty. We will show that the same is true for $i=1$. This is of course the case if $m_1=1$. So assume that $m_1\geq 2$.

\medskip
\noindent Let $F:X_\alpha\to X_\beta$ be a quasi-isometry and denote also by $F$ its boundary extension. By Theorem \ref{localclassesinfty}, $F$ induces a homeomorphism $F_1:\R^n/H^\alpha_1\to \R^n/H^\beta_1$. Let $\pi_\alpha$ and $\pi_\beta$ be the canonical projections onto $\R^n/H^\alpha_1$ and $\R^n/H^\beta_1$ respectively. The linear actions verify
$$\pi_\alpha\circ \mathrm{Exp}(t\alpha)=\mathrm{Exp}(t\alpha_1)\circ\pi_\alpha\text{ and }\pi_\alpha\circ \mathrm{Exp}(t\beta)=\mathrm{Exp}(t\beta_1)\circ\pi_\beta,$$
where now the biggest Jordan block of the first eigenvalue is $m_1-1$ (notice that $k_1^{\alpha_1}=k_1^\alpha$ and $k_1^{\beta_1}=k_1^\beta$). Let $\varrho_\alpha$ and $\varrho_\beta$ be the parabolic visual metrics defined before. The key point to apply the induction argument is to notice that the left cosets of $H_1^\alpha$ and $H^\beta_1$ are equidistant for $\varrho_\alpha$ and $\varrho_\beta$ respectively. Let $\tilde{\varrho}_\alpha$ and $\tilde{\varrho}_\beta$ be the quotient distances. On checks, see \cite[Lemma 15.9]{JeremyTyson01}, that $F_1$ is a quasi-symmetry in these metrics. Since they are left invariant and the linear actions of $\alpha_1$ and $\beta_1$ are dilations with dilation expansion $e^t$, they are bi-Lipschitz equivalent to the respective parabolic visual metrics $\varrho_{\alpha_1}$ and $\varrho_{\beta_1}$. This shows that $(\R^n/H^\alpha_1,\varrho_{\alpha_1})$ is quasi-symmetric to $(\R^n/H^\beta_1,\varrho_{\beta_1})$, and induction applies. 

\medskip
\noindent Note that we have shown, in particular, the case when $d=1$. So suppose that $d\geq 2$ and that the theorem is proved for $d-1$. Considering the quotients by $K_1^\alpha$ and $K_1^\beta$ and applying similar arguments as above, we show by induction that for $i=2,\ldots,d$ we have $k_i^\alpha=k_i^\beta=k_i$ and $m^\alpha_{ij}=m_{ij}^\beta$ for $j=1,\ldots,k_i$. Then we can apply the previous assertion and conclude that $k_1^\alpha=k_1^\beta=k_1$ and $m^\alpha_{1j}=m^\beta_{1j}$ for all $j=1,\ldots,k_1$. This ends the proof.

\bibliographystyle{alpha}
\bibliography{references.bib}
\end{document}